\documentclass[10pt,a4paper]{amsart} 
\usepackage[draft]{optional}
\usepackage[british,american]{babel}
\usepackage[margin=0cm]{geometry}
\usepackage{lipsum}

\usepackage{mathrsfs}

\usepackage{amsfonts, amsmath, wasysym, caption}

\usepackage{amssymb,amsthm, paralist}

\usepackage{amsopn}
\usepackage{accents,bbm,bibgerm,dsfont,eucal,esint,paralist,url,verbatim,wasysym}
\usepackage[normalem]{ulem}
\usepackage{bbold}

\usepackage{
latexsym,
nicefrac,
rotating
}

\usepackage[usenames]{color}
\usepackage{tikz}
\usetikzlibrary{decorations.pathreplacing}
\usepackage{url}

\definecolor{darkgreen}{rgb}{0,0.5,0}
\definecolor{darkred}{rgb}{0.7,0,0}

\usepackage[colorlinks, 
citecolor=darkgreen, linkcolor=darkred
]{hyperref}

\usepackage{esint}
\usepackage{bibgerm}
\usepackage[normalem]{ulem}

\usepackage{enumitem}
\theoremstyle{plain}
\newtheorem{lemma}{Lemma}[section]
\newtheorem{thm}[lemma]{Theorem}

\theoremstyle{definition}

\newtheorem{rmk}[lemma]{Remark}

\newtheorem{ass}[lemma]{Assumption}

\numberwithin{equation}{section}

\newcommand{\al}{\alpha}

\newcommand{\de}{\delta}

\newcommand{\ka}{\kappa}

\newcommand{\La}{\Lambda}

\newcommand{\si}{\sigma}

\newcommand{\Si}{\Sigma}
\renewcommand{\th}{\theta}

\newcommand{\R}{\ensuremath{{\mathbb R}}}

\newcommand{\downto}{\downarrow}
\newcommand{\upto}{\uparrow}

\DeclareMathOperator{\inj}{inj}

\newcommand{\norm}[1]{\Vert#1\Vert}  
\newcommand{\bnorm}[1]{\big\Vert#1\big\Vert}

\def\osc{\mathop{{\mathrm{osc}}}\limits}

\newcommand{\arsinh}{{\rm arsinh}}

\newcommand{\brmk}{\begin{rmk}}
\newcommand{\ermk}{\end{rmk}}
\newcommand{\partref}[1]{\hbox{(\csname @roman\endcsname{\ref{#1}})}}

\newcommand{\fs}{\text{.}}
\newcommand{\cm}{\text{,}}
\newcommand{\Eucl}{\mathrm{E}}
\DeclareMathOperator{\Area}{Area}
\newcommand{\gE}{{g_\Eucl}}
\newcommand{\DD}[1]{\partial_{#1}}

\newcommand{\beq}{\begin{equation}}
\newcommand{\eeq}{\end{equation}}
\newcommand{\beqs}{\begin{equation*}}
\newcommand{\eeqs}{\end{equation*}}
\newcommand{\beqa}{\begin{equation}\begin{aligned}}
\newcommand{\eeqa}{\end{aligned}\end{equation}}
\newcommand{\beqas}{\begin{equation*}\begin{aligned}}
\newcommand{\eeqas}{\end{aligned}\end{equation*}}

\newcommand{\half}{\frac{1}{2}}
\newcommand{\thalf}{\tfrac{1}{2}} 
\newcommand{\dist}{\text{\textnormal{dist}}}
\renewcommand{\i}{\mathrm{i}}

\newcommand{\pt}{\partial_t}

\newcommand{\A}{\ensuremath{{\mathcal A}}}

\newcommand{\abs}[1]{{\lvert#1\rvert} }
\newcommand{\babs}[1]{{\big\lvert#1\big\rvert} } 
\newcommand{\eps}{\varepsilon}
\newcommand{\na}{\nabla}

\newcommand{\Hol}{{\mathcal{H}}} 

\newcommand{\Col}{\mathcal{C}}
\newcommand{\thin}{\text{-thin}}

\newcommand{\Ima}{\text{Im}}
\newcommand{\Rea}{\text{\textnormal{Re}}}

\newcommand{\dbar}{\bar\partial}
\newcommand{\dvg}{dv_g}
\newcommand{\BM}{\mathcal{B}}
\newcommand{\Cyl}{{\mathscr{C}}}
\newcommand{\Atild}{{\mathcal{A}}}
\newcommand{\Btild}{{\widetilde{\mathcal{B}}}}
\newcommand{\Leb}{{\mathcal{L}}}
\newcommand{\Nt}{{\widetilde{N}}}
\newcommand{\Nth}{{\widetilde{N_c}}}
\newcommand{\gNt}{{g_{\Nt}}}
\newcommand{\Len}{\mathscr{L}}
\title
[{\sc
Finite-time degeneration for variants of Teichm\"uller harmonic map flow }]
{{\sc
Finite-time degeneration for variants \\ of Teichm\"uller harmonic map flow }
\\ 
}
\author{Craig Robertson and Melanie Rupflin}
\date{\today}

\begin{document}
\begin{abstract}
We consider the question of whether solutions of variants of Teichm\"uller harmonic map flow from surfaces $M$ to general targets can degenerate in finite time. For the original flow from closed surfaces of genus at least $2$, as well as the flow from cylinders, we prove that such a finite-time degeneration must occur in situations where the image of thin collars is `stretching out' at a rate of at least $\inj(M,g)^{-(\frac14+\de)}$, and we construct targets in which the  flow from cylinders must indeed degenerate in finite time. For the rescaled Teichm\"uller harmonic map flow, the condition that the image stretches out is not only sufficient but also necessary and we prove the following sharp result: Solutions of the rescaled flow cannot degenerate in finite time if the image stretches out at a rate of no more than $\abs{\log(\inj(M,g))}^{\frac12}$, but must degenerate in finite time if it stretches out at a rate of  at least $\abs{\log(\inj(M,g))}^{\frac12+\de}$ for some $\de>0$.
\end{abstract}

\maketitle

\section{Introduction and Results}
Teichm\"uller harmonic map flow is a geometric flow, introduced by Topping and the second author in \cite{RT},
that is designed to flow surfaces to minimal surfaces. 
Given any closed orientable surface $M$  and any compact Riemannian manifold $(N,g_N)$, the flow evolves both a map $u\colon M\to N$ and a constant curvature metric $g$ on the domain $M$ (with $K_g\equiv 1,0,-1$ for surfaces of genus $\gamma=0,1, \geq 2$) by the gradient flow of the Dirichlet energy 
$E(u,g)=\frac12\int_M\abs{du}_{g}^2 \, dv_g.$
Solutions of the flow are characterised by 
\beq 
\label{eq:flow}
\pt u=\tau_g(u), \qquad \pt g=\frac{\eta^2}{4}P_g(\Rea(\Phi(u,g)))\eeq
where $\tau_g(u)$ is the tension of $u$, $\Phi(u,g)$ the Hopf differential, $P_g$ denotes the $L^2$-orthogonal projection onto the space that is formed by the real parts of holomorphic quadratic differentials and $\eta>0$ is a fixed coupling constant. We refer to \cite{RT} for more detail on the definition of the flow. 
The results from the joint works \cite{RT,Rexistence,HRT,RT-global} of Topping (respectively Huxol and Topping) and the second author
establish that for any initial data, the flow \eqref{eq:flow} admits a global weak solution which decomposes the initial map into a collection of minimal surfaces. This solution is smooth away from finitely many singular times and we distinguish between two different types of possible finite-time singularities: If the 
injectivity radius remains bounded away from zero as $t$ approaches a singular time $T<\infty$, then the metric component of the flow remains regular across $T$ while the singularity of the map component is caused by the bubbling off of a finite collection of harmonic spheres, a phenomenon that is well known already from the theory of classical harmonic map flow, c.f.~\cite{Struwe85}. Here we are concerned with the second type of possible finite-time singularities in which the metric degenerates due to the injectivity radius of $(M,g)$ tending to zero as $t$ approaches a finite time $T$.

We remark that this second type of finite-time singularities is excluded if the domain is a torus, in which case an equivalent flow was already introduced and studied by Ding, Li and Liu in \cite{Ding-Li-Liu}, since the Teichm\"uller space of the torus (equipped with the Weil--Petersson metric) is complete. 

For higher genus surfaces, the delicate analysis in \cite{RT-neg} guarantees that if the target $(N,g_N)$ has non-positive curvature, or more generally if the target supports no bubbles, then solutions of \eqref{eq:flow} cannot degenerate in finite time. 
This naturally raises the question of whether solutions of the flow can degenerate in finite time at all and if so, what mechanism can lead to such a finite-time degeneration. 

Before we formulate our answers to the above questions, we remark that a careful analysis of the flow at such potential finite-time degenerations has been carried out in the joint works \cite{ RT-horizontal, RT-global} of Topping and the second author and recall the following properties of solutions $(u,g)$ of the flow: If $\inj(M,g(t))\to 0$ as $t\upto T$ for some $T<\infty$, then there are a finite number of geodesics $\si_j(t)$ in $(M,g(t))$, with $\si_j(t)$ homotopic to $\si_j(t')$, whose length $\ell_j(t)=L_{g(t)}(\si_j(t))\to 0$ as $t\upto T$. These geodesics are surrounded by long thin collar neighbourhoods $\Col(\si_j(t))$
 which are isometric to cylinders $\big([-X(\ell_j(t)),X(\ell_j(t))]\times S^1, \rho^2(s) (ds^2+d\th^2)\big)$, compare Lemma \ref{lemma:collar} in the appendix. On the thick part of these collars, and indeed on the thick part of the whole surface, the metric converges as described in \cite{RT-horizontal, RT-global} and so does the map component, up to the usual bubbling off of harmonic spheres. The  results of \cite{RT-global} also ensure that away from finitely many regions where further bubbles may form, the map component $u$  maps the degenerating part of the collar very close to curves in the sense that the oscillation of $u$ over circles $\{s\}\times S^1$ tends to zero, see \cite{RT-global} for details. As already seen in \cite{RT-neg} we furthermore know that the behaviour of the map component $u$ away from the degenerating part of a collar has hardly any influence on the evolution of the length of the central geodesic of such a thin collar.

 This raises the question of what behaviour of the map component on the degenerating parts of a collar can be responsible for a degeneration of the metric. 
As we shall see, the driving force for a finite-time degeneration of the flow is not merely the formation of a bubble on a very thin collar in the domain but rather a very stretched-out image in the target of such a thin collar. 
It is hence the behaviour of the map component on the long cylinders connecting different bubble regions on the degenerating part of the collar (or connecting them to the thick part of the surface) that can force the flow to degenerate. To be more precise,  
we shall see that if the length of one of the curves in the target close to which these cylinders are mapped is large compared to $\ell^{-\frac14}$, $\ell$ the length of the central geodesic of the corresponding collar in the domain, then finite-time degeneration must occur. 
Such a phenomenon can e.g.~be caused by the formation of a winding bubble on the collar as we shall discuss in detail in Section \ref{section:cyl}, compare also \cite{T-winding}.

As a measure of how long these curves in the image are we shall consider
\beq \Len(u,\Col(\si)):=\int_{-X(\ell)}^{X(\ell)} \fint_{S^1} \abs{u_s} \, d\th \, ds, \label{def:L} \eeq
 $(s,\th)$ collar coordinates on the collar $\Col(\si)$ around a simple closed geodesic $\si$, c.f.~Lemma \ref{lemma:collar}. 

Our first  main result establishes that degeneration in finite time must occur if $\Len(u,\Col(\si))$ is of order $\ell^{-(\frac14 +\de)}$ for some $\de>0$.
We stress that we do not ask that $\Len(u,\Col(\si))$ becomes unbounded, which would already amount to imposing a priori that the solution degenerates in finite time, but just ask for a suitable relation between $\Len(u,\Col(\si))$ and the length of the central geodesic of the corresponding collar, and of course that $\ell$ is below some fixed threshold. For suitable targets, such as the warped products that we construct in Section \ref{section:cyl}, these properties can be proven to hold true for all maps with suitable symmetries and energy bounds, which will allow us to obtain examples of finite-time degeneration in Section \ref{section:cyl}.

Our main results, which we state in Theorems \ref{thm:1} and \ref{thm:2} below,  apply not only for solutions of the flow \eqref{eq:flow}, respectively \eqref{eq:flow-resc}, on closed surfaces of genus at least $2$, but also for solutions  on cylinders. We recall the relevant properties of the flow \eqref{eq:flow} from cylinders introduced in \cite{R-cyl}, which is designed to flow to minimal surfaces with prescribed boundary curves, in Appendix \ref{appendix:cyl} and note that in this case one considers \eqref{eq:flow} together with Plateau boundary conditions. As our analysis will be carried out in the interior of the domain, the specific nature of the boundary conditions is  however irrelevant in our proofs, and all we will require of solutions of \eqref{eq:flow} is that they
satisfy an energy inequality of the form 
\beq \label{ass:energy-ineq} 
E(u(t_1),g(t_1))-E(u(t_2),g(t_2))\geq c \int_{t_1}^{t_2}\norm{\partial_t u}_{L^2(M,g)}^2+\norm{\partial_t g}_{L^2(M,g)}^2\, dt \text{ for some } c>0
\eeq
for almost every $t_1<t_2$, as expected for a gradient flow, and, in case $N$ is non-compact, that for almost every $t$ the image of $u(t)$ is bounded. 
We note that these properties are satisfied 
in particular for the solutions obtained in \cite{Rexistence, RT-global,R-cyl}, and note that for the (unique) weak solutions of the flow on closed domains which have non-increasing energy we indeed have 
\beq \label{eq:energy-decay}
 \frac{dE}{dt}=-\int_M 
|\tau_g(u)|^2+\left(\frac{\eta}{4}\right)^2 |P_g(\Rea(\Phi(u,g)))|^2 \,\dvg,
\eeq
away from finitely many singular times.

\begin{thm} \label{thm:1}
Let $M$ be either a closed oriented surface of genus $\gamma\geq 2$ or a cylinder, let $(N,g_N)$ be a complete Riemannian manifold with sectional curvature bounded from above by some 
$\bar\kappa\in \R$, and let $E_0<\infty$, $c_0>0$, and $\delta\in(0,1)$ be any fixed numbers. Then there exists a number $\bar \ell>0$ 	depending only on $E_0$, $c_0$, $\de$, the genus of $M$ (respectively the fixed constant $d$ in the definition \eqref{def:Gell} of the flow on cylinders) and $\bar\kappa$  so that the following holds true:
\newline
Let $(u,g)$ be a weak solution of \eqref{eq:flow} with
$E(u,g)\leq E_0$ which satisfies the energy inequality \eqref{ass:energy-ineq} and which is defined on an interval $[0,T)$ that is maximal in the sense that $T=\infty$ or $\inf_{t\in[0,T)}\inj(M,g(t))=0$. 
Suppose that there exist a homotopically non-trivial simple closed curve $\si_0$ in $M$, a number $ \eps_1>0$, and a time $T_1<T$ so that  for every $t\in[T_1,T)$
\begin{enumerate}[label=(\roman*),ref=\roman*,nosep]
\item \label{ass:ell_0}
the length $\ell(t)$ of the simple closed geodesic $\si(t)$ homotopic to $\si_0$ satisfies $\ell(t) \leq \bar \ell$,\\
and 
\item \label{ass:tens_gives_length}
for those $t\in [T_1,T)$ for which $\norm{\tau_g(u)(t)}_{L^2(M,g(t))}\leq \eps_1 \cm
$ 
we  furthermore have that
\beqs 
\Len(u(t),\Col(\si(t))\geq c_0 \, \ell(t)^{-\frac{1}{4}(1+\delta)},\qquad  \text{ where } \Len \text{ is defined in \eqref{def:L}}. \eeqs
\end{enumerate}
Then $T<\infty$, i.e.~the solution $(u,g)$ of the flow \eqref{eq:flow} must degenerate in finite time. 
\end{thm}

Here we call $(u,g)$ a weak solution of \eqref{eq:flow} if  $g$ is a Lipschitz curve of hyperbolic metrics that satisfies the second equation in \eqref{eq:flow} for almost every $t$ and 
$u\in H^1([0,T)\times M,N) $ (respectively $u\in H^1_{loc}$ if $T=\infty$) is
 a weak solution of the first equation in \eqref{eq:flow}, where we may assume without loss of generality that $N$ is embedded in some Euclidean space by Nash's embedding theorem.

While it would be interesting to know if the rate $\ell^{-(\frac14+\de)}$ can be improved, a much more important  question is whether the stretching out of collars is not only a sufficient but also a necessary condition for finite-time degeneration, or conversely if there are solutions of \eqref{eq:flow} that degenerate in finite time but for which the quantity $\Len$ introduced in \eqref{def:L} remains bounded. 

For weak solutions of the rescaled flow 
\beq 
\label{eq:flow-resc}
0=\tau_g(u), \qquad \pt g=P_g(\Rea(\Phi(u,g))),\eeq
which was introduced by Huxol in \cite{Huxol}, we shall prove that the flow degenerates if and only if the image is stretching out, and indeed that finite-time degeneration is excluded if $\Len(u,\Col(\si))$ is controlled by 
$\log(\ell^{-1})^\half$, while any growth rate of $\Len(u,\Col(\si))$ of order $\log(\ell^{-1})^{\half+\delta}$, $\de>0$, will force the solution to degenerate in finite time. 
Here we say that $(u,g)$ is a weak solution of \eqref{eq:flow-resc} if the maps $u(t)$ are harmonic with bounded energy and $g$ is a Lipschitz curve of hyperbolic metrics for which \eqref{eq:flow-resc} is satisfied for almost every time. 
\begin{thm}\label{thm:2}
Let $M$ be either a closed oriented surface of genus $\gamma\geq 2$ or a cylinder and let $(N,g_N)$ be a complete Riemannian manifold with sectional curvature bounded from above.
\newline
Let $(u,g)$ be a weak solution of \eqref{eq:flow-resc} defined on an interval $[0,T)$ which is maximal in the sense that $T=\infty$ or $\inf_{[0,T)} \inj(M,g(t))=0$. Then the following hold true: 
\begin{enumerate}[label=(\roman*),ref=\roman*,nosep]
\item \label{item:thm2-1} Suppose that there exist
$C_1<\infty$ and $\bar \ell\in(0,1)$ so that
for every $t\in[0,T)$ and every simple closed geodesic $\si(t)\subset (M,g(t))$ 
with length $\ell(t)=L_{g(t)}(\si(t))\leq \bar \ell$
we have 
\beqs 
\Len(u(t), \si(t))\leq  C_1\log(\ell(t)^{-1})^\half.
\eeqs
Then $T=\infty$, i.e.~the flow cannot degenerate in finite time.
\item \label{item:thm2-2} Conversely for any numbers  $\de, c_0>0$ there exists a constant $\bar \ell\in(0,1)$ (with the same dependencies as in Theorem \ref{thm:1}),
so that the following holds true: 
\newline
Suppose that there 
exists a homotopically non-trivial simple closed curve $\si_0\subset M$ and a number $T_1<T$ so that on 
$[T_1,T)$ we have $\ell(t)=L_{g(t)}(\si(t))\leq \bar \ell$ and 
\beqs 
\Len(u(t), \Col(\si(t)))\geq c_0\log(\ell(t)^{-1})^{\half(1+\de)},\eeqs
$\si(t)$ the unique geodesic in $(M,g(t))$ homotopic to $\si_0$. 
Then $T<\infty$, i.e.~the flow must degenerate in finite time.

\end{enumerate}
\end{thm}

In Section \ref{section:cyl} we will construct settings in which 
our main results apply, thus establishing that finite-time degeneration does indeed occur for both variants of the flow  from cylinders.

\begin{thm} \label{thm:cyl}
Let $(N,g_N)$ be a complete Riemannian manifold of bounded sectional curvature as constructed in  Section \ref{section:cyl} and let  
$u_0 \colon [-1,1]\times S^1\to N$ and $g_0$ be any pair of initial map and metric as described in Lemma \ref{lemma:initial}. 
Then the solution of \eqref{eq:flow}  degenerates in finite time.\newline
Furthermore, choosing $(N_c,g_{N_c})$ to be a compact 
Riemannian manifold as constructed in the last part of Section \ref{section:cyl} we have that any weak solution of the rescaled flow \eqref{eq:flow-resc} with map component as described in Remark \ref{rmk:flow-rescaled} must degenerate in finite time.
\end{thm}

\textbf{Acknowledgement:}
The first author was supported by Engineering and Physical Sciences Research Council [EP/L015811/1]. 

\section{Proof\ of Theorem \ref{thm:1}}
In this section we prove our main result for solutions of the flow \eqref{eq:flow}. We initially carry out this proof for smooth solutions and explain at the end of the section what minor modifications are needed to prove the result also for weak solutions.

We first recall that the evolution of the length of a simple closed geodesic under Teichm\"uller harmonic map flow is essentially determined by a weighted integral of the Hopf differential $\Phi(u,g)=\big( \abs{u_s}^2-\abs{u_\th}^2-2\i \langle u_s,u_\th\rangle\big)(ds+\i d\th)^2$ of $u$ over the corresponding collar neighbourhood $\Col(\si)$.
To be more precise, let $(u,g)$ be a solution of \eqref{eq:flow} defined on either a closed surface of genus at least 2 or on a cylinder
and let 
$\si_0$ be a simple closed curve in $M$ which is homotopically non-trivial.
If $M$ is a closed surface of genus at least $2$, then 
Lemma 2.3 of \cite{RT-neg}  asserts that the evolution of
$\ell(t)=L_{g(t)}(\si(t))$, $\si(t)$ the unique geodesic  homotopic to $\si_0$,  under \eqref{eq:flow} is so that
\beq  \label{est:ddl} \bigg|\frac{d}{dt}\log(\ell^{-1})-\frac{\eta^2}{16\pi^3}\cdot \ell \int_{-X(\ell)}^{X(\ell)}\int_{S^1} \left(\abs{u_s}^2-\abs{u_\theta}^2\right)\rho^{-2} \, d\th \, ds \bigg|
\leq C \ell \, \eta^2 E_0 , \eeq
$(s,\th)\in[-X(\ell),X(\ell)]\times S^1$ collar coordinates on $\Col(\si)\subset M$,  in which the metric is given by $g=\rho^2(ds^2+d\th^2)$, compare Lemma \ref{lemma:collar}.  Here $C$ depends only on the genus of the surface.

As remarked in Lemma 4.4 of \cite{R-cyl} and its proof,  the same estimate applies also if $M$ is a cylinder, 
now for a constant $C$ that only depends on the fixed parameter $d>0$ from the definition \eqref{def:Gell}  of the metric component of the flow on cylinders, compare Appendix \ref{appendix:cyl}.
For cylinders, the collar $\Col(\si)$ agrees with the whole cylinder and is isometric to 
 $\big([-X(\ell),X(\ell)]\times S^1,\rho^2(ds^2+d\th^2)\big)$ where $\rho$ is still given by \eqref{eq:collar-general} but now $X(\ell)=\tfrac{2\pi}{\ell}(\tfrac\pi2-\arctan(\tfrac\ell d))$.

To treat both situations at the same time, in the following we consider maps $u \colon \Col(\ell)\to (N,g_N)$ that are defined on hyperbolic cylinders of the form
\beq\label{eq:collar-general}
\big(\Col(\ell),g\big) \cong \big([-X(\ell),X(\ell)]\times S^1, \rho^2(ds^2+d\th^2)\big), \quad \rho(s)=\frac{\ell}{2\pi}\cos(\frac{\ell}{2\pi}s)^{-1}
\eeq
where we ask that $X(\ell)>0$ is so that for some $\bar c_{1,2}>0$
\beq \label{def:constants-collar}
\frac{2\pi}{\ell}\big(\frac\pi2- \bar c_1\ell\big) \geq X(\ell) \geq \frac{2\pi}{\ell}\big(\frac\pi2-\bar c_2\ell\big) \quad \text{if} \quad \ell\in (0, \arsinh(1)).
\eeq We note that for collars in closed hyperbolic surfaces 
\eqref{def:constants-collar} holds true for universal constants
$\bar c_{1,2}>0$, while for cylinders the constants $\bar c_{1,2}$ depend only on the fixed constant $d>0$ used in the definition of the metrics \eqref{def:Gell}, compare Appendix \ref{appendix:cyl} and \cite{R-boundary}.

As a first step towards the proof of Theorem \ref{thm:1}, we show that for any map from such a hyperbolic cylinder, the weighted integral of the Hopf differential appearing in \eqref{est:ddl} is bounded from below by the following estimate.

\begin{lemma}\label{lemma:I_lower_general}
Let $(N,g_N)$ be a complete Riemannian manifold with sectional curvature bounded from above by some $\bar \kappa\in\R$. Then for any smooth map $u \colon \Col=\Col(\ell)\to (N,g_N)$ from a hyperbolic cylinder $(\Col,g)$ as in \eqref{eq:collar-general} and \eqref{def:constants-collar} with $\ell\in (0,\arsinh(1))$, and any \mbox{$J \subset [-X(\ell),X(\ell)]$}, we have
\beq \label{est:I_lower_general}
\ell\int_{J \times S^1} \left(\abs{u_s}^2-\abs{u_\theta}^2\right)\rho^{-2} \, d\th \, ds \geq -C \cdot (1+\norm{\tau_g(u)}_{L^2(\Col,g)}^2) \cdot (1+\log(\ell^{-1}))
\eeq
where $C$ depends only on  $\bar\ka$, an upper bound $E_0$ on the energy of $u$ and $\bar c_{1,2}>0$ from \eqref{def:constants-collar}.
\end{lemma}

Combining Lemma \ref{lemma:I_lower_general} with \eqref{est:ddl} hence tells us that for solutions of \eqref{eq:flow} we may always bound 
\beq \label{est:ddl_weak_lower} 
\frac{d}{dt}\log(\ell^{-1})\geq -C\cdot (1+\norm{\tau_g(u)}_{L^2(\Col,g)}^2)\cdot [1+\log(\ell^{-1})]-C\ell.
\eeq
We shall apply this very weak lower bound at those times $t$ where we do not have good control on the tension and hence, in the setting of Theorem \ref{thm:1}, do not impose a lower bound on  $\Len(u,\Col)$. 
Conversely, for maps $u$ for which the assumption \eqref{ass:tens_gives_length} of Theorem \ref{thm:1} is satisfied, we obtain a far stronger bound.

\begin{lemma} \label{lemma:I_lower_strong}
Suppose that in the setting of Lemma \ref{lemma:I_lower_general} we have additionally that 
\beq \label{ass:L_lower_lemma}
\Len(u,\Col(\ell)):=\int_{-X(\ell)}^{X(\ell)} \fint_{S^1} \abs{u_s} \, d\th \, ds \geq c_0 \, \ell^{-\frac{1}{4}(1+\de)}, 
\eeq
for some $c_0>0$ and some $\de\in(0,1)$.
Then
\beq 
\label{claim:lemma22}
\ell \int_\mathcal{C} \left( |u_s|^2 - |u_\theta|^2 \right) \rho^{-2} \, d\th \, ds\geq c_1 \ell^{-\delta} - C\left( 1 + \| \tau_g(u) \|_{L^2(\Col,g)}^2 \right) \left( 1 + \log ( \ell^{-1} ) \right) \eeq
for some $c_1>0$ and $C\in\R$ that depend only on $c_0$, $\de$ and as always $\bar\ka$, $E_0$  and $\bar c_{1,2}>0$. 
\end{lemma}
We can now combine Lemma \ref{lemma:I_lower_strong} with \eqref{est:ddl} to conclude that for  solutions $(u,g)$ of \eqref{eq:flow} and times $t$ with $\norm{\tau_{g(t)}(u(t))}_{L^2(M,g(t))}
\leq 1$ we have 
\beq \label{est:ddl_weak_stronger1} 
\frac{d}{dt}\log(\ell^{-1})\geq c_1 \, \ell^{-\de} -C\cdot [1+\log(\ell^{-1})]-C\ell. 
\eeq
for collars $\Col(\si(t))$ around geodesics of length $\ell(t)\in (0,\arsinh(1))$ on which 
 \eqref{ass:L_lower_lemma} holds. 
Given any numbers $c_0>0$ and $\de\in(0,1)$,
we can and will choose $\bar \ell\in(0,1)$ sufficiently small so that for $\ell\in(0,\bar \ell)$ the first term in the above expression dominates, say so that 
$\frac{c_1}2 \ell^{-\delta} \geq C [1 + \log(\ell^{-1})]+C\ell $. Hence \eqref{est:ddl_weak_stronger1} implies that 
if $(u,g)$ is a solution of \eqref{eq:flow}, then 
 \beq \label{est:ddl_weak_stronger2} 
\frac{d}{dt}\log(\ell^{-1})\geq \frac{c_1}{2} \, \ell^{-\de}  \text{ whenever } 
\eqref{ass:L_lower_lemma} \text{ holds, } 
\ell\in (0,\bar \ell) \text{ and } \norm{\tau_g(u)}_{L^2(M,g)}\leq 1.
\eeq

We shall give the proof of these two lemmas later and first 
explain how they can be used to prove Theorem \ref{thm:1}.

\begin{proof}[Proof of Theorem \ref{thm:1}]
Let $(u,g)$ and $\si(t)$ be as in Theorem \ref{thm:1} and let $\ell(t)=L_{g(t)}(\si(t))$.
We let 
$ 
D := \left\{t\in[T_1,T) : \norm{\tau_g(u)(t)}_{L^2(M,g(t))}\leq \eps_1\right\} 
$
be the set of times where the second assumption of the theorem 
ensures that \eqref{ass:L_lower_lemma} holds on $\Col(\si(t))$. As we may 
assume without loss of generality that $\eps_1\leq 1$ and that $\bar \ell\in (0,1)$ is chosen as above we may 
 hence bound $\frac{d}{dt}\log(\ell^{-1})$ from below according to \eqref{est:ddl_weak_stronger2} on $D$.
Conversely, for times $t\notin D$ we bound $\frac{d}{dt}\log(\ell^{-1})$  using \eqref{est:ddl_weak_lower}. Combined we obtain that 
 $f(t) := 1 + \log(\ell(t)^{-1}) $ satisfies a differential inequality of the form
\beq \label{est:ddl_proof}
f'(t) \geq \frac{c_1}2 \ell(t)^{-\delta} \, \mathbb{1}_{D}-g(t) f(t) = c_2 \, e^{\delta f(t)} \, \mathbb{1}_{D}-g(t) f(t) \text{ on } [T_1,T) .
\eeq
Here $g = C \cdot \mathbb{1}_{D^c} \cdot (1+\norm{\tau_g(u)}_{L^2(\Col,g)})$, which we note is integrable over $[T_1,T)$, even if $T=\infty$, since the energy inequality \eqref{ass:energy-ineq} assures that 
$\mathcal{L}^1(D^c) \leq  CE_0 \,\eps_1^{-2}<\infty$.

Hence $f(t)\leq F(t) := \exp\big( \int_{T_1}^t g(t') \, dt' \big) \cdot f(t)  \leq e^{\norm{g}_{L^1}} f(t) $ on  $ [T_1,T)$,  which allows us to conclude from \eqref{est:ddl_proof} that
\beqs 
F'(t) \geq c_2 \, e^{\de f(t)} \, \mathbb{1}_{D} \geq c_2 \, e^{\tilde{\de} F(t)} \, \mathbb{1}_{D}, 
\eeqs
where $\tilde{\de} = \de e^{-\norm{g}_{L^1}} > 0 $.
Integration over $[T_1,t)$ then yields 
\beqs 0 < e^{-\tilde{\de}F(t)} \leq e^{-\tilde{\de}F(T_1)} - c_2 \tilde \de\cdot \mathcal{L}^1(D \cap [T_1,t) ) \leq 1 -c_2 \tilde \de\cdot (t-T_1-CE_0 \, \eps_1^{-2})\eeqs
which leads to a contradiction for $t$ sufficiently large. Hence we cannot have that $T=\infty$, so the solution has to degenerate in finite time.
\end{proof}
To prove the above Lemmas \ref{lemma:I_lower_general} and \ref{lemma:I_lower_strong} we shall use the following standard estimates on the angular energy and on the $H^2$-norm of maps from Euclidean cylinders on 
regions with low energy, compare e.g.~\cite{DT, HRT, Struwe85}. Here and in the following we write for short $\Cyl_\La(s):=[s-\La,s+\La]\times S^1$ and $\Cyl_\La=\Cyl_\La(0)$ and will equip these cylinders with the Euclidean metric $\gE=ds^2+d\th^2$ unless specified otherwise.

\begin{lemma}[compare e.g.~\cite{DT, HRT, Struwe85}] 
\label{lemma:angular}
Let $(N,g_N)$ be a complete Riemannian manifold with sectional curvature bounded from above by some $\bar \kappa\in\R$. 
Then there exist numbers $\eps_0=\eps_0(\bar\kappa)>0$ and $C=C(\bar\kappa)\in\R$ so that the following holds true for $H^2$ maps $u \colon \Cyl_X\to (N,g_N)$ away from 
\beqs 
 \mathcal{A} := \{ s: \abs{s}\leq X-1 \text{ and } E(u;\Cyl_1(s)) \geq \varepsilon_0 \} \cup [-X,-X+1]\cup [X-1,X].\eeqs
Let $\varphi_0\in C_c^\infty((-1,1),\R^+)$ be a fixed cut-off function with $\varphi_0\equiv 1$ on $[-\half,\half]$ and $\abs{\varphi_0'}+\abs{\varphi_0''}\leq C$. Then for any $s_0\notin \A$ we have that 
\beq\label{est:H2}
\int \varphi_0(\cdot-s_0)^2 \, \abs{\na^2 u}^2 \, d\th \, ds \leq C E(u;\Cyl_1(s_0))+C\int\varphi_0(\cdot-s_0)^2 \, \abs{\tau_\gE(u)}^2 \, d\th \, ds
\eeq
where $\na^2 u$ is the (intrinsic) Hessian of the map $u \colon (\Cyl,\gE)\to (N,g_N)$, i.e.~$\na^2 u=\na du$ is computed using the connection on $T^*\Cyl\otimes u^*TN$ induced by the Levi-Civita connections.
\newline
Furthermore, the angular energy 
$ \vartheta(s) = \int_{\{s\} \times S^1} |u_\theta|^2 \, d\theta $
satisfies
\beq \label{est:angular} \vartheta(s) \leq C E_0 \exp(-\dist_\gE (s,\mathcal{A})) + C \int_{\Cyl_X} |\tau_\gE(u)|^2 \, e^{-|s-q|} \, dq \, d\theta , \eeq
for every $s$ with $\dist_\gE (s,\mathcal{A})\geq 1$, 
$E_0$ an upper bound on the energy of $u$. 
\end{lemma}

We recall that the tension with respect to the hyperbolic metric $\rho^2(ds^2+d\th^2)$  is related to the Euclidean tension by $\tau_g(u)=\rho^{-2}\, \tau_\gE(u)$. We also note that $\sup_{\abs{q}\leq X(\ell)}
 e^{-\abs{s-q}/2} \rho(q)\leq C\rho(s)$ for a constant $C$ that depends only on an upper bound on $ \rho$ on the collar, and hence only on $\bar c_{1}$ from \eqref{def:constants-collar} and an upper bound on $\ell$. 
We hence immediately obtain from 
\eqref{est:angular} that for maps from hyperbolic collars as considered in 
Lemmas \ref{lemma:I_lower_general} and \ref{lemma:I_lower_strong}
\beq \label{est:theta-lower} 
\vartheta(s) \leq C \exp(-\dist_\gE (s,\mathcal{A})) + C \rho(s)^2 \int_\mathcal{C} |\tau_g(u)|^2 \, e^{-|s-q|/2} \, dv_g \text{ for } s\notin \mathcal{A}, \eeq
$\mathcal{A}$ as in the above Lemma \ref{lemma:angular}, compare also \cite{RT-neg}.

In the proofs of both Lemmas \ref{lemma:I_lower_general} and \ref{lemma:I_lower_strong} we will encounter weighted integrals of this angular energy and as such it will be useful to split our collar $\Col=\Col(\ell)$ into $\BM\times S^1$, for
\beq \label{def:BM}
 \BM := \left\{ s \in [-X(\ell),X(\ell)] : \dist_\gE(s,\mathcal{A}) \geq 4 \log(\rho(s)^{-1}) + 2 \right\}\eeq
 where \eqref{est:theta-lower} yields strong bounds on angular energies,  
  and $\BM^c\times S^1$, whose  measure is very small compared to $X(\ell)$ if $\ell$ is small, compare Lemma \ref{lemma:BM} below.  

We first note that given any map $u$ from a collar as in Lemma \ref{lemma:I_lower_general}
 or Lemma \ref{lemma:I_lower_strong} and any subset $K\subset \BM$ we can use \eqref{est:theta-lower} to estimate
\beqa \label{est:lower-theta-BM}
\ell \int_{K  \times S^1}\abs{u_\th}^2  \rho^{-2} \, d\theta \, ds 
\leq C\ell \int_{\BM} \bigg(\rho^{2}(s)
+ \int_{\Col}   |\tau_g(u)|^2 \, e^{-|s-q|/2} \, dv_g \bigg) \, ds \leq C\ell (1+\norm{\tau_g(u)}_{L^2(\Col,g)}^2),
\eeqa
where we use in the last step that $\int \rho^2\leq \text{Area}(\Col,g)$ and that the area of the collars we consider is bounded above in terms of only $\bar c_1$.  
In the setting of Lemma \ref{lemma:I_lower_general} we will use this bound to simply estimate 
\beq \label{est:lower-weak-BM}
\ell \int_{(J \cap \BM) \times S^1}\big(\abs{u_s}^2-\abs{u_\th}^2 \big) \, \rho^{-2} \, d\theta \, ds \geq -C\ell (1+\norm{\tau_g(u)}_{L^2(\Col)}^2),
\eeq
while in the setting of Lemma \ref{lemma:I_lower_strong} we will combine the above bound on the angular energy with a much stronger lower bound on  $\int_{\BM \times S^1}\abs{u_s}^2 \rho^{-2}$ that we will later derive from the main assumption \eqref{ass:L_lower_lemma} of that lemma. 

In both cases, we have to consider the high-energy regions $\BM^c$ separately 
from $\BM$ and it will be useful to note that $\BM^c$ has the following simple properties:

\begin{lemma}
\label{lemma:BM}
Let $u$ be a map from a hyperbolic collar $\Col$ as in Lemma \ref{lemma:I_lower_general} with energy bounded by $E_0$, and let $\BM$ be defined by \eqref{def:BM}. 
Then the number of connected components of $\BM^c$ is bounded above by $\frac{E_0}{\eps_0} + 2$, and the length of each connected component $I_j$ of $\BM^c$ is bounded by
\beq\label{est:I_j}
\abs{I_j}\leq C \cdot\big(\sup\nolimits_{I_j}\log(\rho^{-1}) + 1\big),  \eeq
for $C= 8\big(\frac{E_0}{\eps_0} + 2\big)$. Furthermore, for points contained in the same connected component we have that the conformal factor is of comparable size, namely
\beq \label{est:rho-comp-Ij}
\sup\nolimits_{I_j} \rho \leq C \inf\nolimits_{I_j} \rho \text{, and hence in fact }
\abs{I_j}\leq C \cdot (\inf\nolimits_{I_j}\log(\rho^{-1})+1)
\eeq
for a constant $C$ that depends only on $\eps_0$ and $E_0$, and  the constant $\bar c_{1}$ from \eqref{def:constants-collar}.
\end{lemma}

\begin{proof}[Proof of Lemma \ref{lemma:BM}]
The first claim immediately follows from the definition of $\BM$ as each connected component $I_j$ of $\BM^c$ must either contain $\pm X(\ell)$ or an interval of length $2$ with energy at least $\eps_0$. 
Similarly, for each connected component $I_j$ of $\BM^c$ we can choose a maximal family of disjoint intervals $I_j^k\subset I_j$ of length $2$ with energy at least $\eps_0$ and note that the total number of such intervals is bounded by $\frac{E_0}{\eps_0}$. As each point in $I_j$ will have distance at most $4\sup_{I_j} \log(\rho^{-1})+3$ from one of these intervals or from one of the endpoints $\pm X(\ell)$, we hence obtain that \eqref{est:I_j} holds.
To prove \eqref{est:rho-comp-Ij}, we first recall that 
since 
$\abs{\partial_s\log(\rho))}\leq \rho$ we have 
$e^{-1}\rho(s_0)\leq \rho(s)\leq e \rho(s_0)$
for any points $s,s_0\in [-X(\ell),X(\ell)]$ with $\abs{s-s_0}\leq e^{-1} \rho(s_0)^{-1}$, compare \eqref{est:rho-different-points-2}.
If $\sup_{I_j} \rho^{-1}$ is large enough so that 
$\abs{I_j}\leq C (\log(\sup_{I_j} \rho^{-1})+1)$ is smaller then $e^{-1} \sup_{I_j} \rho^{-1}$, we hence obtain that \eqref{est:rho-comp-Ij} holds true for $C=e$. Conversely, for any other connected component $I_j$ of $\BM$ we have a uniform upper bound on $\sup_{I_j} \rho^{-1}$ (depending only on $\eps_0$ and $E_0$) and hence a uniform upper bound on $\abs{I_j}$, so 
\eqref{est:rho-comp-Ij} follows from \eqref{est:rho-different-points-1}. 
\end{proof}

We are now in a position to complete the proofs of Lemmas \ref{lemma:I_lower_general} and \ref{lemma:I_lower_strong}. 
\begin{proof}[Proof of Lemma \ref{lemma:I_lower_general}]
We first note that if $\ell\geq \ell_1$, for some number $\ell_1>0$ determined below,  then \eqref{est:I_lower_general} is trivially satisfied as $\rho^{-2}\leq 4\pi^2\ell_1^{-2}$. 
So assume that $\ell\in(0,\ell_1)$.
In light of \eqref{est:lower-weak-BM}, it is enough to show that for any given subset $J\subset [-X(\ell),X(\ell)]$ and each $j$
$$ 
\ell\int_{(J \cap I_j) \times S^1} \left( |u_s|^2 - |u_\th|^2 \right) \rho^{-2} \, d\theta \, ds\geq  -C \cdot (1+\norm{\tau_g(u)}_{L^2(\Col,g)}^2) \cdot (1+\log(\ell^{-1}))
$$ where $I_j$ are the connected components of the set $\BM^c$ whose properties were discussed in Lemma \ref{lemma:BM} and whose measure satisfies in particular $\Leb^1(\BM^c)\leq C(1+\log(\ell^{-1}))< 2X(\ell)$, provided $\ell_1$ is chosen suitably small. So $\BM$ must be non-empty and hence 
each connected component $I_j$ of $\BM^c$ has at least one endpoint $s_j$ which is an element of $\BM$. 
Writing for short 
$\psi(s):= \int_{\{s\}\times S^1} \abs{u_s}^2-\abs{u_\th}^2 \, d\th$ 
we can therefore bound for each $s\in I_j$
$$\psi(s)\geq \psi(s_j)-\int_{I_j} \abs{\partial_s \psi} \, ds \geq -\vartheta(s_j)-\int_{I_j} \abs{\partial_s \psi} \, ds\geq 
-C\rho^4(s_j) -C\rho^2(s_j)\norm{\tau_g(u)}_{L^2}^2- \int_{I_j} \abs{\partial_s \psi} \, ds, $$
where we applied \eqref{est:theta-lower} in the last step and where norms are computed over the corresponding collar.  
Using \eqref{est:rho-comp-Ij}, 
$\abs{I_j} \leq C(\log(\ell^{-1})+1)$ and that $\ell\rho^{-1}\leq 2\pi$,
 we may hence bound 
\beqa \label{est:nearly-done}
\ell \int_{J \cap I_j} \psi(s)\rho^{-2}(s) \, ds& \geq  -C \ell\abs{I_j}\cdot \rho^{2}(s_j)-C \ell \abs{I_j}\cdot \norm{\tau_g(u)}_{L^2}^2- \abs{I_j} \int_{I_j} \rho^{-1}\abs{\partial_s \psi}\\
&\geq -C(1+\log(\ell^{-1}))\cdot \big[ 1+\norm{\tau_g(u)}_{L^2}^2 +\int_{I_j}\rho^{-1} \abs{\partial_s \psi} \big].
\eeqa
We now complete the proof of Lemma \ref{lemma:I_lower_general} by proving that  $\int\rho^{-1}\abs{\partial_s \psi} \leq C (1+\norm{\tau_g(u)}_{L^2}^2)$. To this end, we note that $\psi$ can be viewed as an integral of the real part of the function $\phi=\abs{u_s}^2-\abs{u_\theta}^2-2\i \langle u_s,u_\theta \rangle$ which represents the Hopf differential $\Phi=\phi \,dz^2$ of $u$ in collar coordinates. As the antiholomorphic derivative of the Hopf differential is controlled in terms of the (Euclidean) tension, in particular
$
\abs{ \partial_s \Rea( \phi )-\partial_\theta \Ima(\phi)
}=| \Rea( \dbar \phi ) | \leq 2 \left| \tau_\gE(u) \right| \left| du \right|_\gE ,
$
we have that 
\beqs
\int_{-X(\ell)}^{X(\ell)} \rho^{-1} \abs{ \partial_s \psi} \, ds \leq 2\int_{\Col} \rho^{-1} \left| \tau_\gE(u) \right| \left| du \right|_\gE d\th \, ds \leq CE_0^\half \norm{\tau_g(u)}_{L^2(\Col,g)}\leq C  (1+\norm{\tau_g(u)}_{L^2(\Col,g)}^2)
\eeqs
where we used $\tau_g(u)=\rho^{-2}\, \tau_\gE(u)$ in the penultimate step. Together with \eqref{est:lower-weak-BM}
and
 \eqref{est:nearly-done} this completes the proof of the lemma. 
\end{proof}

\begin{proof}[Proof of Lemma \ref{lemma:I_lower_strong}]
We first remark that it suffices to consider the case that $\ell\in (0,\ell_2)$ for a constant $\ell_2>0$ chosen below, as for larger values of $\ell$ the claim already follows from Lemma \ref{lemma:I_lower_general}.
To prove the  lower bound \eqref{claim:lemma22} on the weighted integral $\int \rho^{-2} \psi \, ds$,  $\psi:=\int_{\{s\}\times S^1} \abs{u_s}^2-\abs{u_\th}^2 \, d\th$ as above, we 
divide $[-X(\ell),X(\ell)]=K\cup K^c$ into the subset $K\subset \BM$ defined below, and its complement $K^c$ on which we shall simply apply Lemma \ref{lemma:I_lower_general} to  bound
\beq \label{est:int-over-J}
\ell\int_{K^c} \rho^{-2} \psi \, ds  \geq -C \cdot (1+\norm{\tau_g(u)}_{L^2(\Col,g)}^2) \cdot (1+\log(\ell^{-1})) \fs
\eeq
More precisely, we define
$
 K= \{ s \in \BM : \inj_g(s,\cdot)<\eps(\ell) \},$ for $\eps(\ell) := \La \cdot \ell^{\frac12(1+\de)}$ and  $\La>0$ determined below,
which we know is non-empty for 
$\ell\in(0,\ell_2)$ since $\de<1$, provided  $\ell_2$ is sufficiently small. 
We will prove below that this choice of $\eps(\ell)$ ensures that the (Lebesgue) measure of $K^c$ is small enough so that 
the assumption \eqref{ass:L_lower_lemma} implies that also 
\beq \label{est:ass-for-eta} 
\frac{1}{2\pi}\int_{K\times S^1}\abs{u_s} \geq \frac{c_0}{2} \ell^{-\frac{1}{4}(1+\de)}.
\eeq
At the same time the hyperbolic area of $K\times S^1$ is also small, namely
$$\int_{K\times S^1}\rho^2=\Area(K\times S^1,g)\leq \Area(\eps(\ell)\thin(\Col),g)\leq C\eps(\ell) \leq C \ell^{\half(1+\de)},$$ compare e.g.
 \cite [(A.2)]{RTZ}.
Combined, these two estimates imply that 
\beqas
\ell \int_{K \times S^1} |u_s|^2 \, \rho^{-2} \, d\th \, ds 
&\geq \ell \Big(\int_{K \times S^1} \rho^2 \, d\th \, ds \Big) ^{-1} \cdot \Big( \int_{K \times S^1} |u_s| \, d\th \Big)^{2} 
\geq c_1\ell^{-\de}
\eeqas
for some $c_1>0$. 
We note that combining this estimate with \eqref{est:int-over-J} and the estimate on the angular energy on $K$ already obtained in \eqref{est:lower-theta-BM} immediately yields the 
claim of the lemma. 

It hence remains to prove that \eqref{est:ass-for-eta} holds true for a suitably chosen $\La$. To this end we first recall that 
$\eps\thin(\Col)=\{p\in \Col: \inj_g(p)<\eps\}$ is given in collar coordinates by a cylinder $(-X_\eps(\ell),X_\eps(\ell))\times S^1$, for a number $X_\eps(\ell)$ which is such that 
$X(\ell)-X_\eps(\ell)\leq C\eps^{-1}$, compare \eqref{eq:X-delta}.
Combined with the properties of $\BM^c$ obtained in Lemma \ref{lemma:BM} this yields
$$ \mathcal{L}^1(K^c)\leq \sum_j\abs{I_j} +C\eps^{-1} \leq C(\log(\ell^{-1})+1)+C\La^{-1} \ell^{-\frac12(1+\de)}.$$
For $\La$ sufficiently large and
 $\ell_2>0$ sufficiently small we hence obtain that for $\ell\in(0,\ell_2)$
\beqas 
\frac1{2\pi}\int_{K^c\times S^1} \abs{u_s} \, d\th \, ds &\leq C(\mathcal{L}^1(K^c))^{\frac12} \leq C \big[ 
(\log(\ell^{-1}))^\frac12+1+\La^{-\half} \ell^{-\frac14(1+\de)}\big]\leq \frac{c_0}{2} \ell^{-\frac14(1+\de)}.
\eeqas
Combined with \eqref{ass:L_lower_lemma}  this implies the claim \eqref{est:ass-for-eta}, which completes the proof of the lemma.
\end{proof}

\begin{rmk}
\label{rmk:weak-sol-1}
The above proof of Theorem \ref{thm:1} applies with only minor modifications if $(u,g)$ is a weak solution of \eqref{eq:flow}: For almost every $t$ the maps 
$u(t)$ are of class $H^2$ in the interior of $M$ so Lemmas \ref{lemma:angular} and \ref{lemma:BM} may be applied. 
 Lemmas \ref{lemma:I_lower_general} and \ref{lemma:I_lower_strong} also apply without change at these times as their 
 proofs are based only on Lemmas \ref{lemma:angular} and \ref{lemma:BM}.
Since $g$ is Lipschitz, and hence differentiable almost everywhere, we can thus estimate 
$\frac{d\ell}{dt}$ precisely as in the above proof for almost every time, which yields the result.
\end{rmk}

\section{Proof of Theorem \ref{thm:2}}
We now turn to the proof of our second main result which gives a sharp criterion for finite-time degeneration for solutions of the rescaled flow. 
For this we exploit that 
at each time $t$  the map component $u(t)$ of a solution of \eqref{eq:flow-resc} is  a harmonic map,
so its Hopf differential is holomorphic which implies that the quantity $\psi(s)=\int_{\{s\}\times S^1}\abs{u_s}^2-\abs{u_\th}^2 \, d\th$ considered above is constant on each collar around a simple closed geodesic $\si(t)$ in $(M,g(t))$. Knowing that a map is harmonic furthermore allows us to  prove the following stronger relationship between $\psi$ and $\Len(u,\Col(\si))$, which will be the main tool in the proof of Theorem \ref{thm:2}.
\begin{lemma}\label{lemma:psi-harmonic}
For any numbers $\bar \kappa, C_1\in \R$ respectively $\bar \kappa\in\R$ and $\de,c_0>0$ there exist numbers $\bar \ell, C_2>0$, respectively $\bar \ell, c_1>0$ so that the following holds true  
for any complete Riemannian manifold $(N,g_N)$ whose sectional curvature is bounded from above
by $\bar \kappa$ and any hyperbolic collar $\Col(\ell)$ as in \eqref{eq:collar-general}, \eqref{def:constants-collar} for which $\ell\in (0,\bar \ell)$.

Suppose that $u \colon \Col(\ell)\to (N,g_N)$ is a harmonic map 
for which the quantity $\Len$ introduced in \eqref{def:L} is bounded above by $\Len(u,\Col(\ell))\leq C_1(\log(\ell^{-1}))^{\half}$. Then $\psi= \int_{\{s\}\times S^1}\abs{u_s}^2-\abs{u_\th}^2 \, d\th$, which is constant on the collar, is bounded from above by
\beq \label{claim:psi0-upper}
\psi \leq C_2\ell^2[\log(\ell^{-1})+1].
\eeq
Conversely, for harmonic maps $u \colon \Col(\ell)\to (N,g_N)$ with $\Len(u,\Col(\ell))\geq c_0(\log(\ell^{-1}))^{\half+\de}$
we have
\beq
\label{claim:psi0-lower}
\psi\geq c_1\ell^2\log(\ell^{-1}(t))^{1+\de}.
\eeq
\end{lemma}
For the proof of this lemma we 
use the following  variation of the $H^2$ estimate \eqref{est:H2} from Lemma \ref{lemma:angular}, in which we can replace the full energy on the right hand side by only the angular energy. For the sake of completeness, we include a short proof at the end of this section.

\begin{lemma} \label{lemma:second-der}
In the setting of Lemma \ref{lemma:angular} 
we have that for any $s_0\notin\Atild $ %
\beq \label{est:second-der} 
\int_{s_0-\half}^{s_0+\half} \int_{S^1} \abs{\na^2 u}^2 d\th\,ds\leq C \int_{s_0-1}^{s_0+1}\vartheta(s) \, ds + C \, \norm{\tau_{\gE}(u)}^2_{L^2(\Cyl_{2}(s_0),\gE)} \cm
\eeq
where $C$ depends only on an upper bound on the sectional curvature of $(N,g_N)$. 
\end{lemma}

Given a harmonic map $u$ from a hyperbolic collar as in Lemma \ref{lemma:psi-harmonic} we can combine the above lemma with the angular energy estimate of Lemma \ref{lemma:angular} to conclude that for $\ell\in(0,1)$
\beq \label{est:ang-harmonic}
\vartheta(s)+ \int_{s-\half}^{s+\half} \int_{S^1} \abs{\na^2 u}^2 d\th\,ds\leq  C\ell^4
\text{ on } \Btild:= \{s: \dist_{\gE}(s,\Atild)\geq 4\log(\ell^{-1})+2 \}.
\eeq 
We also note that the measure of the complement of $\Btild\subset [-X(\ell),X(\ell)]$ can be bounded as explained in the first part of the proof of Lemma \ref{lemma:BM}, resulting in 
\beq \label{est:meas-Btild-comp}
\Leb^1(\Btild^c)\leq C(\log(\ell^{-1})+1)\leq X(\ell),
\eeq
where the last inequality holds true as $\ell\in (0,\bar \ell)$ for $\bar \ell\in(0,1)$ chosen sufficiently small.

\begin{proof}[Proof of Lemma \ref{lemma:psi-harmonic}]
Let $u$ be as in the lemma. In both cases we compare  $\psi$ to 
$\al(s):= \fint_{\{s\}\times S^1}\abs{u_s} \,d\th,$
and already note that  
\beq \label{est:al}
\al(s)^2\leq \fint_{\{s\}\times S^1}\abs{u_s}^2\,d\th \leq \psi+\vartheta(s) \text{ and  that } \Len(u,\Col(\si))=\int_{-X(\ell)}^{X(\ell)}\al \,ds.
\eeq 
We now estimate
\beqas
 \psi^\half\cdot \abs{\sqrt{2\pi}\al(s)-\psi^\half} &=\frac{\psi^\half}{\sqrt{2\pi}\al(s)+\psi^{\half}} \cdot \abs{2\pi\al(s)^2-\psi}
\leq \int_{\{s\}\times S^1} \abs{\al(s)^2 -\abs{u_s}^2 +\abs{u_\th}^2}\, d\th 
\\
&\leq C \bigg(\int_{\{s\}\times S^1} \babs{\al(s)-\abs{u_s}}^2 \, d\th \bigg)^{\half} \bigg(  \int_{\{s\}\times S^1}  \abs{u_s}^2 \, d\th \bigg)^{\half}+\vartheta(s)\\
&\leq C\bigg( \int_{\{s\}\times S^1}  \abs{\na^2_{s,\th} u}^2 \,d\th\bigg)^{\half}\cdot \big[ \psi+\vartheta(s)\big]^{\half} +\vartheta(s) \cm
\eeqas
where we used \eqref{est:al} in both the penultimate and the last step.
We thus obtain from \eqref{est:ang-harmonic} that 
\beq \label{est:psi0}
\psi^\half\cdot \babs{\psi^\half-\sqrt{2\pi} \int_{s-\half}^{s+\half} \al}\leq C \ell^{2}\cdot [\psi^\half+\ell^{2}] \text{ for every  } s\in\Btild.
\eeq
Suppose now that $\Len(u,\Col(\si)) \leq C_1\log(\ell^{-1})^{\half}$. 
As \eqref{est:meas-Btild-comp} implies that $\mathcal{L}^1(\Btild)\geq X(\ell)\geq c\ell^{-1}>0$, we can now combine \eqref{est:psi0} and \eqref{est:al} to conclude that in this case
\beqa 
\psi= \fint_{\Btild}\psi \,ds & \leq C \psi^{\half}\big[  X(\ell)^{-1} \int_{-
X(\ell)}^{X(\ell)}\al(s) \, ds+\ell^{2}\big] +C\ell^{4} \leq \half \psi+ C\ell^2 \log(\ell^{-1})+C\ell^{4}
\eeqa 
which gives the claimed bound \eqref{claim:psi0-upper} on $\psi$.

It hence remains to prove that if instead $\Len(u, \Col(\si)) \geq c_0\log(\ell^{-1})^{\half(1+\de)}$, then we obtain the lower bound \eqref{claim:psi0-lower} on $\psi$.
To this end, we let $\Btild_{\text{-}}$ be the set of all $s\in \Btild$ for which the interval $[s-\half,s+\half]$ is fully contained in $\Btild$. We note that since the number of connected components of $\Btild$ is bounded uniformly in terms of $E_0$ and $\eps_0$, we have that $\Leb^1(\Btild\setminus \Btild_{\text{-}})\leq C$ and hence  $\Leb^1((\Btild_{\text{-}})^c)\leq C(1+\log(\ell^{-1}))$, compare \eqref{est:meas-Btild-comp}. Thus
\beqs
\int_{(\Btild_{\text{-}})^c}\al \leq C \Leb^1((\Btild_{\text{-}})^c)^{\half} E_0^\half\leq C(\log(\ell^{-1})^{\half}+1)\leq \frac{c_0}2 \log(\ell^{-1})^{\half(1+\de)},
\eeqs
where the last estimate holds as we may assume that $\bar \ell$ was chosen sufficiently small (depending on $c_0,\de>0$). Hence we must have that 
\beq\label{est:al-lower}
 \int_{\Btild_{\text{-}}} \al \geq \frac{c_0}2 \log(\ell^{-1})^{\half(1+\de)}.\eeq
 We first claim that this implies that
$\psi\geq c\ell^{4}>0$.
Indeed if we had $\psi\leq c\ell^{4}$ we would have that $\al(s)\leq \vartheta(s)^{\half} +c\ell^2\leq C\ell^2$ for all $s\in\Btild$ and hence $\int_{\Btild_{\text{-}}} \al\leq \int_{\Btild} \al\leq C\ell$ contradicting \eqref{est:al-lower} if $\bar \ell>0$ is sufficiently small. 

The right hand side of \eqref{est:psi0} 
is hence bounded by $C\ell^2\psi^\half$ for every $s\in\Btild$ and thus
$$\psi^\half\geq \sqrt{2\pi} \Leb(\Btild)^{-1} \int_{\Btild} \int_{s-\half}^{s+\half}\al -C\ell^2 
\geq 
c\ell \int_{\Btild_{\text{-}}} \al-C\ell^2 \geq \sqrt{2c_1} \ell \log(\ell^{-1})^{\half(1+\de)}-C\ell^2$$
for some $c_1>0$, yielding \eqref{claim:psi0-lower}  for $\bar \ell$ sufficiently small.
\end{proof}

We are now in a position to prove our second main theorem.

\begin{proof}[Proof of Theorem \ref{thm:2}]
Suppose first that $(u,g)$ is a smooth solution of the rescaled flow \eqref{eq:flow-resc} and let $\si_0\subset M$ be a homotopically non-trivial simple closed
curve. We denote by $\si(t)$  the geodesic in $(M,g(t))$ homotopic to $\si_0$, and by $\psi(t)$ the corresponding quantity 
considered in Lemma \ref{lemma:psi-harmonic}. We note that $\psi$ is related to the principal part of the Fourier expansion $\Phi(u,g)=\sum b_j e^{jz} \, dz^2$, $z=s+\i\th$ on $\Col(\si(t))$ by $\psi=2\pi\Rea(b_0)$.
We also recall that if $g$ is either a horizontal curve of hyperbolic metrics on a closed surface, i.e.~such that $\partial_t g=\Rea(\Psi(t))$ for holomorphic quadratic differentials $\Psi$, or a curve of metrics on the cylinder as described in Appendix \ref{appendix:cyl}, then  $\ell(t)=L_{g(t)}(\si(t))$  evolves by 
$\frac{d\ell}{dt}=-\frac{2\pi^2}{\ell}\Rea(b_0(\Psi))$, see e.g.~\cite{Wolpert}. In the present situation hence 
\beq \label{eq:length-evol-harmonic}
\frac{d}{dt} \log(\ell^{-1})=\pi\ell^{-2}\cdot \psi
\eeq
which allows us to derive the theorem from Lemma \ref{lemma:psi-harmonic} as follows: 

Suppose first that 
$(u,g)$ is as in part \eqref{item:thm2-1} of the theorem and let $\bar \ell>0$ and $C_2$ be as in Lemma \ref{lemma:psi-harmonic}. This lemma yields that at any time $t\in [0,T)$ 
for which $\ell(t)=L_{g(t)}(\si(t))\in(0,\bar \ell)$
we have 
$\psi(t)\leq C_2\ell(t)^2[\log(\ell^{-1}(t))+1]$ on $\Col(\si(t))$,
and hence also 
$$\frac{d}{dt} (\log(\ell^{-1}(t))+1)\leq C_2 \pi
 (\log(\ell^{-1}(t))+1)
$$
which excludes the possibility that $\ell(t)\to 0$ in finite time. As this argument applies for every non-trivial simple closed curve in $M$, we hence find that the injectivity radius cannot go to zero in finite time, i.e.~that a finite-time degeneration of the metric is excluded. 

Conversely, if $(u,g)$ and $\si(t)$ are as in part \eqref{item:thm2-2} of the theorem for $\bar \ell$ chosen as in Lemma \ref{lemma:psi-harmonic}, then this lemma yields that for every $t\in [T_1,T)$
we have $
\psi(t)\geq c_1 \ell^2\log(\ell^{-1}(t))^{1+\de}$
on $\Col(\si(t))$ 
for a  number $c_1>0$ that is independent of time.
Inserting this lower bound on $\psi$ into \eqref{eq:length-evol-harmonic} then immediately gives that 
 $$\frac{d}{dt} \log(\ell^{-1}(t))\geq  c_1\pi\log(\ell^{-1})^{1+\de} \text{ for } t\in [T_1,T) .
$$
As $\de>0
$, the solution of the flow must hence degenerate in finite time. 

If $(u,g)$ is only a weak solution, the above proof still applies because the maps are still harmonic and hence smooth in the interior, so we can directly apply Lemma \ref{lemma:second-der} at every time, and obtain the same estimates as above at every time at which 
the curve of metrics is differentiable.
\end{proof}

We finally provide a proof of Lemma \ref{lemma:second-der} which follows by well-known arguments that have been used in particular in \cite{DT, Struwe85} to establish $H^2$ bounds such as \eqref{est:H2} in low-energy regions. 

\begin{proof}[Proof of Lemma \ref{lemma:second-der}]
Given any $s_0\notin \Atild$ we let $\varphi=\varphi_0(\cdot-s_0)$, where $\varphi_0$ is a cut-off function as in Lemma \ref{lemma:angular}. As $\na_{s,s}^2u=\tau(u)-\na_{\th,\th}^2u$, where we write for short $\tau=\tau_\gE$, we have
\beqas
I &:= \int\varphi^2\abs{\na^2u}^2 = \int\varphi^2 \Big( \abs{ \tau(u) - \na^2_{\th,\th}u }^2+ \abs{\na^2_{\th,\th}u}^2  + 2\abs{\na^2_{s,\th}u}^2 \Big) \,  \\
&\phantom{:}\leq 
\bnorm{\varphi\tau(u)}_{L^2}^2 +2\norm{\varphi\tau(u)}_{L^2}\cdot \big(\int\varphi^2 \abs{\na_{\th,\th}^2u}^2 \big)^\half+ 2\int\varphi^2 \abs{\na^2_{\th,\th}u}^2\\
&\quad + 2\int\varphi \abs{\varphi'} \abs{u_\th} \, \abs{\na_{s,\th}^2 u} + 2 \bar \kappa \int \varphi^2 \abs{u_\th}^2 \, \abs{u_s}^2 + 2\int\varphi^2 \na^2_{\th,\th}u\cdot \na^2_{s,s}u,
\eeqas 
where $\bar \kappa\geq 0$ is an upper bound on the sectional curvature of the target.
As the last term is equal to $- 2\int\varphi^2 \abs{\na^2_{\th,\th}u}^2+2 \int \varphi^2 \tau(u) \na^2_{\th,\th}u $, we hence obtain that 
\beq  \label{est:remainder}
I\leq \norm{\varphi \tau(u)}_{L^2}^2+C I^\half\cdot \big[\norm{\varphi \tau(u)}_{L^2}+\hat\vartheta(s_0)^\half\big]
+R \leq C\norm{\varphi \tau(u)}_{L^2}^2+ C \hat\vartheta(s_0)+\frac14 I+R
\eeq
where we write for short $\hat\vartheta(s_0)= \int_{s_0-1}^{s_0+1}\vartheta(s) ds$ and set $R:= 2\bar\kappa \int \varphi^2 \abs{u_\th}^2 \abs{u_s}^2$. 
Using that $W_0^{1,1}( [-1,1]\times S^1)$ embeds continuously into $L^2( [-1,1]\times S^1)$ we may now estimate
\beqas
R &\leq C\bnorm{\varphi\abs{u_s}^2}_{L^2} \cdot
\bnorm{\varphi\abs{u_\th}^2}_{L^2}
\leq C\bnorm{\na(\varphi\abs{u_s}^2)}_{L^1}  \cdot
\bnorm{\na(\varphi\abs{u_\th}^2)}_{L^1}\\
 &\leq C\big[ E(u;\Cyl_1(s_0)) +E(u;\Cyl_1(s_0))^\half \norm{\varphi\na^2 u }_{L^2}\big] \cdot \big[ \hat\vartheta(s_0)+\hat\vartheta(s_0)^\half \norm{ \varphi \na^2u}_{L^2} \big]\\
&\leq C\eps_0 \hat\vartheta(s_0)+CI^{\half}\cdot \big[ \eps_0\hat\vartheta(s_0)^\half+\eps_0^{\half} \hat\vartheta(s_0)+C\eps_0^\half\hat\vartheta(s_0)^\half I^\half]
\\
 &\leq 
\frac14 I+C\hat\vartheta(s_0) +C\hat\vartheta(s_0)  I.\eeqas
We can now use the $H^2$ estimate \eqref{est:H2} from Lemma \ref{lemma:angular} to control the last term in this estimate by 
$C\hat\vartheta(s_0)  I\leq C \hat\vartheta(s_0) (\eps_0+\norm{ \varphi\tau(u)}_{L^2}^2)\leq C\hat\vartheta(s_0) +C\norm{\varphi\tau(u)}_{L^2}^2$ and insert the resulting bound on $R$ into \eqref{est:remainder} to obtain the claim of the lemma. 
\end{proof}

\section{Proof of Theorem \ref{thm:cyl} }\label{section:cyl}
In this section we construct settings in which Theorem \ref{thm:1} and Theorem \ref{thm:2} assure that solutions of Teichm\"uller harmonic map flow \eqref{eq:flow},  respectively of the rescaled flow \eqref{eq:flow-resc}, from the cylinder $\Cyl=[-1,1]\times S^1$ degenerate in finite time.
This smooth target will be obtained
 as a warped product  $N=\R \times_f \Nt$
of $\R$ with 
\beq \label{def:g_N_tilde}
(\Nt,g_\Nt):= (\R^3, \rho_\Nt^2 \gE), \text{ where }
\rho_\Nt^2=  C_{\Nt} \exp \left(1-\frac{1}{1-|x|^2}\right)\mathbb{1}_{\{|x|<1\}} + 1, 
\eeq
for $C_\Nt>0$  a (large) constant that we choose below.  Before we analyse solutions $u=(v,w)\colon\Cyl\times [0,T]\to \R \times_f \Nt$ of the flow, or even just discuss the choice of the warping function $f$, we discuss some key properties of maps 
$w\colon\Cyl\to \Nt$. To this end we first note that for large $C_\Nt$ the metric $\gNt$, which is Euclidean outside the Euclidean unit ball but highly concentrated in $\{\abs{x}_\Eucl<1\}$,  is so that the function $r\mapsto r^2 \rho_{\Nt}^2(r)$ has exactly two local extrema $0<r_\mathrm{max}<r_\mathrm{min}<1$ on $[0,1]$, where
$r_\mathrm{max}\downto \frac{\sqrt5-1}2$ and $r_\mathrm{min}\upto 1$ 
 as $C_{\Nt}\to \infty$. 
 
 We will later choose the $w$ component of our initial map $u_0=(v_0,w_0)$ of the flow so that its image is disjoint from the region $\{\abs{x}_\Eucl<1\}$ and first prove the following lower bound on the area of maps $w \colon \Cyl\to (\Nt,\gNt)$ which are obtained by continuously deforming such a map. 

\begin{lemma} \label{lemma:disjoint}
For any $E_0<\infty$ and $z_0>1$ and for  sufficiently large  $C_{\Nt}$ (depending on $E_0$ and $z_0$) we have that the following holds true for the metric $g_{\Nt}$ defined in \eqref{def:g_N_tilde}:
\newline
Let $(w_t)_{t\in[0,T)}$ be any continuous family of smooth, rotationally symmetric maps 
$
w_t \colon \Cyl:=[-1,1]\times S^1\to \R^3 $, 
 $  w_t(x,\th) = \big(r_t(x)e^{\i\th},z_t(x)\big)$,
with fixed boundary values $r_t(\pm1) = r_0\in (0,z_0]$  and  $z_t(\pm1) = \pm z_0$. Suppose that $\Area_{(\Nt, g_{\Nt})}(w_t)\leq E_0$ for every $t\in [0,T)$ and that the image of the initial map $w_0$ is disjoint from the Euclidean open unit ball. Then for every $t\in [0,T)$
\beq \label{claim:lemma-disj-1}
\Area_{(\Nt, g_{\Nt})} (w_t) \geq 2\pi. 
\eeq 
Furthermore, denoting by $\psi(p)$ the angle between $p$ and the positive $z$-axis and by $\abs{p}$ the Euclidean distance to the origin, we  have that for every $t\in [0,T)$
\beq  \label{claim:disjoint}
w_t(\Cyl)\cap \{ p\in\R^3: \abs{p}=r_\mathrm{max} \text{ and }\psi(p)\in [\tfrac{\pi}{4},\tfrac{3\pi}{4}]\}=\emptyset.
\eeq
\end{lemma}
\begin{proof}
We will first explain how \eqref{claim:disjoint} implies \eqref{claim:lemma-disj-1} (for $C_\Nt$ sufficiently large), and provide the proof of \eqref{claim:disjoint} at the end.  
We  note that \eqref{claim:disjoint} allows us to conclude that  
 for each $t \in [0,T)$, there must be an interval $[x_1^t,x_2^t] \subset [-1,1]$ such that 
 $$|w_t| > r_\mathrm{max} \text{ and } \psi(w_t)\in (\tfrac{\pi}{4},\tfrac{3\pi}{4})  \text{ on } (x_1^t,x_2^t) \times S^1\text{ with  } \psi(w_t)( \{x_1^t,x_2^t\}) = \{\tfrac{\pi}{4},\tfrac{3\pi}{4}\}.$$
To see this, we note that the number of such intervals for the initial $w_0$ must be odd as by assumption $\abs{w_0}\geq 1>r_\mathrm{max}$ on all of $\Cyl$  and as the 
points $w_t(\{\pm1\}\times S^1)$ of the fixed boundary circles have $\psi$ coordinate in $(0,\frac\pi4]$ respectively $[\frac{3\pi}{4},\pi)$ since $r_0\leq z_0$. 
Combined with \eqref{claim:disjoint} this property of the boundary data ensures also that along a family of maps $(w_t)$ as in the lemma such intervals can only be lost or gained in pairs, so their  number remains odd, and so nonzero.

The claim \eqref{claim:lemma-disj-1} hence follows if we prove a lower bound of $2\pi$ for the area of any given surface $S_\gamma$ that is obtained by 
rotating a curve $\gamma=\abs{\gamma}e^{\i\psi}\in \R^2$, parametrised over some $[a,b]$, 
around the $z$-axis, for which $\psi(a)=\frac{\pi}{4}$, $\psi(b)=\frac{3\pi}{4}$ and $\abs{\gamma}\geq r_\mathrm{max}$ as well as $\psi\in (\frac{\pi}{4},\frac{3\pi}{4})$ on $ (a,b)$. 
Given such a $\gamma$, say parametrised by arclength and hence with  $\abs{\gamma} \cdot \abs{ \psi'} \leq 1$, we can  estimate
\beqas 
\Area_{(\Nt,\gNt)}(S_\gamma)&=
2\pi \int_{a}^{b} \rho_{\Nt}^2\circ\gamma \cdot |\gamma| \sin (\psi) \geq 2\pi \min_{[a,b]}\big(|\gamma|^2  \rho_{\Nt}^2\circ \gamma\big)  \int_a^b \sin (\psi) \abs{\psi'} \\
&\geq 4\pi  r_\mathrm{min}^2\,
\rho_{\Nt}(r_\mathrm{min})^2   \cos(\frac{\pi}{4}) \geq 2\sqrt{2} \pi  r_\mathrm{min}^2\geq 2\pi
\eeqas
for $C_\Nt$ sufficiently large, where we used that the minimum of $r\mapsto r^2 \rho_{\Nt}^2(r)$ on $[r_\mathrm{max},\infty)$ is achieved at $r_\mathrm{min}$ and that $r_\mathrm{min}\upto 1$ as $C_{\Nt}\to \infty$.

Finally, to prove the claim  \eqref{claim:disjoint} we show that for
$C_\Nt$ chosen sufficiently large, the area of any rotationally symmetric surface with the given boundary conditions 
that violates 
\eqref{claim:disjoint} 
 must have area larger than $E_0$. 
So let  $\gamma=\abs{\gamma} e^{\i \psi}$ be any smooth curve in $\R^2$, say parametrised by arclength over some $[a,b]$,
 with $\gamma(a)= (-z_0,r_0)$ and $\gamma(b)= (z_0, r_0)$ and suppose that there exists $t_0\in (a,b)$ so that $\gamma(t_0)=r_\mathrm{max} e^{\i\psi_1}$ for some $\psi_1\in [\frac{\pi}{4},\frac{3\pi}{4}]$. 

Let now $\eps:=\min(z_0-1,\frac1{10})>0$ which ensures in particular that $[t_0-\eps,t_0+\eps]\subset [a,b]$. 
Since $r_\mathrm{max}\downto \frac{\sqrt{5}-1}{2}\approx
0.62$ as $C_\Nt\to\infty$, we may assume that $C_\Nt$ is large enough to ensure that  
$\half+\frac1{10}\leq r_\mathrm{max}\leq \frac34-\frac1{10}$ so that the above choice of $\eps$ furthermore implies that
$$\thalf\leq  r_\mathrm{max}-\eps\leq \abs{\gamma}\leq r_\mathrm{max}+\eps\leq \tfrac34
\text{ and }
\psi\in [\tfrac{\pi}{4}-2\eps,\tfrac{3\pi}{4}+2\eps]\subset [\tfrac{\pi}{6},\tfrac{5\pi}{6}] \text{ on } [t_0-\eps,t_0+\eps].$$ 
We hence obtain 
\beqas 
\Area_{(\Nt,\gNt)}(S_\gamma)&\geq \Area_{(\Nt,\gNt)}(S_{\gamma\vert_{[t_0-\eps,t_0+\eps]}})= 
2\pi \int_{t_0-\eps}^{t_0+\eps} 
\rho_{\Nt}^2\circ\gamma\cdot |\gamma| \sin (\psi)
\\
&\geq 4\pi \eps C_{\Nt} \exp \big(1-(1-(\tfrac34)^2)^{-1}\big)\cdot \tfrac{1}{4} 
\eeqas
which yields that $\Area_{(\Nt,\gNt)}(S_\gamma)>E_0$ 
provided $C_\Nt=C_\Nt(E_0,z_0)$ is chosen sufficiently large. This completes the proof of the claim \eqref{claim:disjoint} and thus of the lemma. 
\end{proof}
We now construct our target as a warped product 
\beq \label{def:warped_prod} 
N = \R \times_f \Nt \text{ with metric } g_N = dv^2 + f(v)\cdot g_{\Nt},\eeq
where the warping function is of the form $f=f_0(\cdot - \bar v )$,  $\bar v>0$ chosen later, for $f_0$ as in
\begin{ass}\label{ass:f}
We let $f_0 \colon \R \to [1,\infty)$ be
a non-increasing function that satisfies 
$
f_0\equiv 8 \text{ on } (-\infty,0]$  and $f_0>7$  on  $(-\infty,1)$
for which $-f_0'$ is decreasing on $[1,\infty)$ with 
$0<-f_0'\leq \frac18$ and that has  one of the following two types of asymptotic behaviours: 
\begin{enumerate}[label=(\roman*),ref=\roman*,nosep]
\item \label{ass:f-poly}
For some $\de\in(0,1)$ and $c_3,\La>0$ we have that 
$$ 
f_0(v)= 1 + c_3 v^{-(\frac2{1+\de}-1)} \text{ on  } [\La,\infty)  .
$$
\item \label{ass:f-exp}
For some $\al,c_3,\La>0$ we have that 
 \beqs 
f_0(v)= 1 + c_3 e^{-\al (v-\La)}  \text{ on  } [\La,\infty) .
\eeqs
\end{enumerate}
\end{ass}

We observe that in both cases the resulting target $(N,g_N)$ is complete and has bounded curvature, where by construction the curvature bound $\bar \kappa$ is independent of the choice of $\bar v$ and will use functions with asymptotics \eqref{ass:f-poly} to prove the first part of Theorem \ref{thm:cyl}, and functions with asymptotics \eqref{ass:f-exp} in the proof of the second part of the theorem.

In the following we consider maps $u = (v,w) \colon \Cyl=[-1,1]\times S^1 \to N$ with symmetries
\beq \label{def:symms} v = v(x)=v(-x) \cm \qquad w(x,\th) = \big(r(x)e^{\i\th},z(x)\big) \text{ with }  r(x)=r(-x) \text{ and } z(-x)=-z(x)\eeq
and fixed boundary data of
\beq \label{def:bdry}
v(\pm1)=0 \text{ and } r(\pm 1)=\frac14, \quad z(\pm1)=\pm z_0>1. 
\eeq
In general, to evolve a given initial map $u_0$ from the cylinder towards a minimal surface with prescribed boundary curves, one needs to consider \eqref{eq:flow} together with Plateau boundary condition. However, if our initial map 
 $u_0=(v_0,w_0)$ is of the above form and if we choose the  initial metric $g_0=G_{\bar \ell}$ on the cylinder as described in \eqref{def:Gell}, then due to the symmetries this problem is reduced to solving \eqref{eq:flow} with \emph{Dirichlet} boundary conditions
\beq u(\pm1,\th) = u_0(\pm1,\th) \eeq
and the evolution of the metric 
reduces to an ODE  for the length $\ell(t)$ of the central geodesic with $g(t) = G_{\ell(t)}$ as described in \eqref{def:Gell}. Thus
 standard parabolic theory yields the existence of a (unique) solution $(u,g)$ of the flow \eqref{eq:flow} 
 along which the energy decays according to \eqref{eq:energy-decay}. 
 We note that this solution remains smooth for all times unless 
$\inj(\Cyl,g(t))\to 0$ as $t\upto T<\infty$, as a singularity of just the map component, which would need to be caused by the bubbling off of \emph{finitely} many harmonic spheres, is excluded by the symmetries.

We will thus be able to conclude that the flow degenerates in finite time provided we can establish that the assumptions \eqref{ass:ell_0} and \eqref{ass:tens_gives_length} of Theorem \ref{thm:1} hold. 
To this end we fix the constant $C_{\Nt}$ in the definition of $(\Nt,\gNt)$ so that Lemma \ref{lemma:disjoint} applies for maps of area no more than $10\pi$ and choose the initial data as follows.

\begin{lemma} \label{lemma:initial}
Let $(N,g_N)$ be as in \eqref{def:warped_prod}, where $\bar{v}>0$ is any fixed number. Then there is a smooth map $u_0 = (v_0,w_0) \colon \Cyl \to N$ with the symmetries \eqref{def:symms} and boundary data \eqref{def:bdry} that satisfies $\abs{w_0}\geq 1$ on all of $\Cyl$, and a metric $g_0 = G_{\bar \ell}$ as defined in \eqref{def:Gell} such that $E(u_0,g_0) \leq 10\pi.$
\end{lemma}
\begin{proof}
For $\eps>0$ and $\ell>0$ determined below, we 
construct such a map 
$$u_0=(v_0,w_0) \colon (\Cyl,G_\ell)\cong \big([-X(\ell),X(\ell)]\times S^1, \rho_\ell^2 (ds^2+d\th^2)\big) \to (N,g_N)$$ as follows:
We first let $\La_1(\eps)>1$ be the unique number for which $S^2 \setminus B_{\eps}(P^{\pm})$, $P^\pm:=(0,0,\pm 1)$, is conformal to the cylinder $[-\La_1(\eps)+1, \La_1(\eps)-1]\times S^1$ and choose $w_0$ on this central part of the collar to be such a conformal parametrisation, which we can assume to have the symmetries as in \eqref{def:symms}, and which is chosen so that $\{\pm (\La_1(\eps)-1)\}\times S^1$ is mapped to the corresponding boundary curve $S^2\cap \partial B_{\eps}(P^\pm)\subset \R^3$.

We then let $\La_2(\eps)$ 
be the unique number so that 
we can choose  $w_0$ on $ [X(\ell)-(\La_2(\eps)-1),X(\ell)]\times S^1$ to be a conformal parametrisation of the annulus $\overline{D}_{\frac14} \setminus D_{\eps}\times \{z_0\}$, again with the required symmetry and boundary conditions \eqref{def:symms} and \eqref{def:bdry}. 
Here we of course assume that $\ell=\ell(\eps)>0$ is small enough so that $\La_1(\eps)+\La_2(\eps)-2\leq X(\ell)$, but furthermore ask that $\ell$ is sufficiently small so that choosing
$w_0\vert_{[\La_1(\eps),X(\ell)-\La_2(\eps)]\times S^1}(s,\th):= (0,0,1+\frac{s-\La_1(\eps)}{X(\ell)-\La_1(\eps)-\La_2(\eps)} (z_0-1))$ as 
 a linear parametrisation of the line connecting $P^+$ and $(0,0,z_0)$ gives energy 
$$E(w\vert_{[\La_1(\eps),X(\ell)-\La_2(\eps)]\times S^1})= \pi (z_0-1)^2(X(\ell)-\La_1(\eps)-\La_2(\eps))^{-1}\leq \eps.$$ 
We finally complete $w_0$ to a smooth map with the desired symmetries which can be done so that 
on the intermediate regions $[\La_1(\eps)-1,\La_1(\eps)]\times S^1$ and $[X(\ell)- \La_2(\eps), X(\ell) -(\La_2(\eps)-1)]\times S^1$ we have  $\abs{\na w_0}\leq C\eps$ for some universal $C>0$. 

To construct a suitable $v$ component for the initial data, we first fix $v^*=v^*(\bar v)$ so that $f(v^*)=2$. We then ask that 
$v_0 \colon [-X(\ell),X(\ell)]\times S^1 \to \R $ is a smooth map with $v_0(s,\th)=v_0(s)=v_0(-s)$ so that 
$v_0(\pm X(\ell)) = 0$,  
$v_0\equiv v^*$ on the central part $|s| < \La_1(\eps)$, and  
 $|v_0'| \leq 2\frac{v^*}{X(\ell)-\La_1(\eps)}$.
For $\ell=\ell(\eps,\bar v)>0$ small enough, we hence obtain that also 
$E(v_0)=2\pi\frac{4 (v_0^*)^2}{X(\ell)-\La_1(\eps)}\leq \eps$. In total we hence obtain an energy of no more than 
\beqas E(u_0,g_0) &\leq    2 \, \max_\R f \cdot \Area(D_{\frac14} \setminus D_{\eps}) + f(v_0^*) \Area(S^2) + C\eps \leq 9\pi   + C\eps \leq 10 \pi \eeqas 
provided we initially chose $\eps>0$ sufficiently small.
\end{proof}

We note that this choice of initial data assures that along the flow the energy of  $u(t) \colon (\Cyl,g(t))\to (N,g_N)$ remains below $10\pi$. 
As the warping function $f$ satisfies $f\geq 1$ everywhere, we thus obtain in particular an upper bound  of $10\pi$ on the energy, and thus on the area, of the component $w(t) \colon (\Cyl,g(t))\to (\Nt,\gNt)$ allowing us to apply Lemma \ref{lemma:disjoint} (since we chose $C_\Nt$ accordingly). 
We hence know that along the flow the component $w(t)$ remains disjoint from the set described in \eqref{claim:disjoint} and hence in particular that points on the central curve $\{s=0\}$ are mapped to points which have Euclidean distance at least $r\geq r_\mathrm{max}$ from the $z$-axis.

While Lemma \ref{lemma:disjoint} already yields a lower bound of $2\pi$ on the area, and thus on the energy, of the whole second component $w$, we now prove that already the restriction of $w$ to a small central region of the domain will need to have at least energy $2\pi$.

\begin{lemma}\label{lemma:energy-centre}
Let $w \colon \Cyl=[-1,1]\times S^1\to (\Nt,\gNt)$ be any map with symmetries and boundary conditions as in \eqref{def:symms} and \eqref{def:bdry} which satisfies \eqref{claim:disjoint} 
and which is so that $\abs{w}\geq r_\mathrm{max}$ on ${\{0\}\times S^1}$. 
Let now $G_\ell$ be any metric on $\Cyl$ as described in \eqref{def:Gell} and  let $(s,\th)\in[-X(\ell),X(\ell)]\times S^1$ be the corresponding collar coordinates, where we assume that $\ell$ is small enough so that $X(\ell)\geq 8$. Then we have a lower bound on 
the energy of the restriction of $w$ to the central part of
\beq 
\label{est:energy-centre} 
E(w\vert_{\{(s,\th): \abs{s}\leq 8\}})=\half\int_{-8}^8\int_{S^1} (\abs{w_s}^2+\abs{w_\th}^2) \, \rho_\Nt^2 \, d\th \, ds \geq 2\pi.\eeq
\end{lemma}
\begin{proof}
If the image of $w\vert_{\{(s,\th): \abs{s}\leq 8\}}$ intersects the cone $\{p: \psi(p)=\frac{\pi}4 \text{ or } \psi(p)=\frac{3\pi}{4}\}$ then we know from the proof of Lemma \ref{lemma:disjoint} that the area of the image of this map in $(\Nt,\gNt)$ is at least $2\pi$ and hence so is its energy. Otherwise, the image of $w\vert_{\{(s,\th): \abs{s}\leq 8\}}$ is contained in $\{p: \abs{p}\geq r_\mathrm{max} \text{ and } \psi(p)\in (\frac\pi4,\frac{3\pi}{4})\}$ which assures in particular that the radial component of $w(s,\th)=(r(s)e^{\i\th},z(s))$ is at least $r(s)\geq r_\mathrm{max}\cdot \sin(\frac\pi4)$. As $\rho_\Nt\geq 1$, we hence obtain that 
$$E(w\vert_{\{(s,\th): \abs{s}\leq 8\}})\geq \half \int_{-8}^{8}\int_{S^1}\abs{w_\th}^2 \, d\th \, ds \geq 16\pi\cdot \half r_\mathrm{max}^2> 2\pi, \text{ as } r_\mathrm{max}>\half.\qedhere$$
\end{proof}
To decrease the energy of the map $u=(v,w)$ into the warped product $\R\times_f (\Nt,\gNt)$, 
the flow  thus wants to stretch out the $v$ component so that the energy of $w$ on this central part is weighted with a small warping factor $f(v)$. This stretched-out image of the $v$ component will in turn drive the degeneration of the metric as proven in  Theorem \ref{thm:1}. 

We will obtain the required lower bound on $\Len(u(t),\Col(\si(t)))$  in 
Lemma \ref{lemma:leash_length}, where we will use in particular that 
the $v$ component must map the central region to the part of the line where  $-f'$ is decreasing. 

\begin{lemma} \label{lemma:v-initial-bound}
Let $(N,g_N)$ be defined as in \eqref{def:warped_prod}  where $f=f_0(\cdot-{\bar v})$ for some $\bar v\geq 7$ and some smooth function $f$ satisfying Assumption \ref{ass:f}. Let $(u_0,g_0)$ be initial data as obtained in Lemma \ref{lemma:initial} and let $(u=(v,w),g=G_{\ell})$ be the corresponding solution of \eqref{eq:flow}. 
Then on the whole existence interval we have  $\ell\leq \ell_2(\bar v):=\frac{5\pi^2}{\bar v^{2}}
 $ and furthermore
\beq 
\label{est:v-centre} 
v(s)\geq \bar v +1 \text{ for every } \abs{s}\leq 8\eeq
 where $(s,\th)$ are the collar coordinates of $(\Cyl,g)$.

\end{lemma}

\begin{proof}
Lemma \ref{lemma:disjoint} assures that at any time 
the energy of $w \colon (\Cyl,g)\to (\Nt,\gNt)$ is at least $2\pi$ and hence 
$10\pi\geq E(u,g)\geq \min_{\Cyl} f(v)\cdot  2\pi=2\pi f(\max v)$. This ensures that $f(\max v)\leq 5$ and hence $\max v\geq \bar v+1$. 
As $v(\pm X(\ell))=0$ we hence get
\beqs 
2\bar v\leq 2 \max v\leq \int_{-X(\ell)}^{X(\ell)} |v'| \, ds \leq (\pi^{-1} E(u,g))^{\half}\cdot (2X(\ell))^{\half}\leq (20 X(\ell))^{\half},
\eeqs
and thus $\frac{\pi^2}{\ell}\geq X(\ell)\geq \frac15 \bar v^{2}$, which implies the first claim of the lemma. 
This estimate ensures in particular that $X(\ell)\geq 8$ for $\bar v\geq 7$. By 
 Lemma \ref{lemma:energy-centre}, thus already the restriction of $w$ to the central region $\{(s,\th): \abs{s}\leq 8\}$ must have energy at least $2\pi$ and hence
$\min_{[-8,8]}f\circ v\leq 5$. As $\abs{f'}\leq \frac18$ and $v(x)=v(-x)$ we conclude that
$\osc_{[-8,8]}f\circ v\leq \frac18 \int_0^8\abs{v'} \, ds\leq \frac18\cdot\big(\frac{1}{2\pi} E(u,g)\big)^{\half} \cdot \sqrt{8} \leq  \sqrt{\tfrac{5}{8}}.$
Therefore $\max_{[-8,8]}f(v(s))\leq 6$ and hence $v(s)\geq \bar v+1$ for every $\abs{s}\leq 8$ as claimed.
\end{proof}

We are now in a position to prove that the image of maps as considered above stretches out as required in the setting of our main results.

\begin{lemma} \label{lemma:leash_length}
For any $c_3,\La>0$ and $\de\in ( 0,1)$, respectively $c_3,\La,\al>0$, 
 there  exist $\ell_1 , c_0 > 0$ such that the following holds true:
\newline
Let $(N,g_N)$ be defined as in \eqref{def:warped_prod}
for a warping function $f=f_0(\cdot-\bar v)$, $\bar v\geq 7$,  where $f_0$  satisfies Assumption \ref{ass:f} either with decay \eqref{ass:f-poly} (for $c_3,\La,\de$)  or decay \eqref{ass:f-exp} (for $c_3,\La,\al$).
Let 
 $u \colon (\Cyl,g=G_\ell) \to (N,g_N)$ be a map with tension 
$\| \tau_g(u) \|_{L^2(\Cyl,g)} \leq 1$ and energy 
 $E(u,g) \leq 10\pi$
for which \eqref{est:energy-centre} and \eqref{est:v-centre}  hold and suppose that $\ell\leq \ell_1$.
Then in the case that $f_0$ has polynomial decay
\beq \label{eq:leash_delta_1} v_\mathrm{max} := \max_{\Col(\ell)} v \geq c_0 \, \ell^{-\frac14(1+\de)}, \eeq
while in the case that
 $f_0$ has exponential decay
\beq \label{eq:leash_delta_2} v_\mathrm{max} := \max_{\Col(\ell)} v \geq c_0 \log (\ell^{-1}). \eeq
\end{lemma}

\begin{proof}
The  $v$ component of the tension field is given in collar coordinates $(s,\th)$ by 
\beq \label{eq:v_tens} \tau_g(u)^v  = \rho^{-2}\DD{ss} v - f'(v) \, e_g(w) \eeq
where $e_g(w)=\half \abs{dw}_g^2$ denotes the energy density of $w \colon (\Cyl,g)\to (\Nt,\gNt)$. 
As the symmetry of $v$ ensures that $\DD{s} v(0)=0$ we can integrate \eqref{eq:v_tens} over $([0,s]\times S^1,\rho^2(ds^2+d\th^2))$ to obtain that
\beqas 
- 2\pi \DD{s} v(s) &= - \int_0^s\int_{S^1}\DD{ss} v  \rho^{-2} \, dv_g 
= - \int_0^s \int_{S^1}  f'(v) \, e_g(w) + \tau_g(u)^v \, dv_g  \\
&\geq -f'(v_\mathrm{max}) E(w\vert_{[0,\min(s,8)]\times S^1}) - 
\Area([0,s]\times S^1, g)^{\half}
\eeqas
where we use that $\norm{\tau_g(u)}_{L^2(\Cyl,g)}\leq 1$ and that
$v([-8,8])$ is contained in $[\bar v+1,\infty)$  where $-f'$ is decreasing.  
Combined with \eqref{est:energy-centre} and \eqref{est:int_rho_by_ell}  we hence conclude that 
$$-\DD{s} v(s)\geq -\thalf f'(v_\mathrm{max})-\rho^{\half}(s)(2\pi)^{-\half} \text{ for } s\geq 8$$
and $-\DD{s} v(s)\geq -\rho^{\half}(s)(2\pi)^{-\half}$
for  $\abs{s}\leq 8$. 
Using 
\eqref{est:int_rho_by_ell} as well as that $X(\ell)\geq \frac{\pi^2}{\ell}-C$ we get 
\beqa \label{est:vmax-main}
v_\mathrm{max}&\geq -\thalf(X(\ell)-8) f'(v_\mathrm{max})-(2\pi)^{-\half}\int_0^{X(\ell)} \rho^{\half}(s) \, ds
\geq  -4\ell^{-1} f'(v_\mathrm{max})-4 \ell^{-\half}\\
&=-4\ell^{-1} f_0'(v_\mathrm{max}-\bar v)-4 \ell^{-\half}\geq 4(-f_0'(v_\mathrm{max})\ell^{-1}-\ell^{-\half})
,\eeqa
where we use that  $\ell\in (0,\ell_1)$ for 
a sufficiently small $\ell_1$ 
and that $-f'_0$ is decreasing on 
$ [1,\infty)$. 

 We first note that for $\ell_1\in(0,1)$ chosen sufficiently small we must have that 
$v_\mathrm{max}\geq \La$. 
Indeed, as $v_\mathrm{max} \geq \bar v + 1 \geq 1$ we would otherwise have that $-f'_0(v_\mathrm{max})\geq -f_0'(\La)=c>0$, where $c=c_3(\frac{2}{1+\de}-1)\La^{-\frac{2}{1+\de}}$ respectively $c=c_3\al$. Hence  
 \eqref{est:vmax-main} would imply that 
$\La> v_\mathrm{max}\geq 4 c \ell^{-1}-4\ell^{-\half}$
which leads to a contradiction as $\ell\leq \ell_1$ for  sufficiently small $\ell_1$. 

Suppose first that $f_0$ as in part \eqref{ass:f-poly} of  Assumption \ref{ass:f}.
If $-\thalf\ell^{-1}f_0'(v_\mathrm{max})\geq \ell^{-\half}$ then \eqref{est:vmax-main} gives
$$v_\mathrm{max}\geq -2 \ell^{-1}f_0'(v_\mathrm{max})=  2
c_3(\tfrac{2}{1+\de}-1)
v_\mathrm{max}^{-\frac{2}{1+\de}} \ell^{-1}, $$
which implies 
that $v_\mathrm{max}\geq c_0 \ell^{-\frac14(1+\de)}$ for some $c_0=c_0(\de,c_3)>0$  as 
$\frac4{1+\de} \geq 1 + \frac2{1+\de}$. This establishes the claim \eqref{eq:leash_delta_1} in this case. 
Conversely, if $-\thalf\ell^{-1}f_0'(v_\mathrm{max})\leq \ell^{-\half}$ then we get 
$$2\ell^\half\geq -f_0'(v_\mathrm{max})= 
c_3(\tfrac{2}{1+\de}-1)
v_\mathrm{max}^{-\frac{2}{1+\de}}$$
which again implies that \eqref{eq:leash_delta_1} holds true for some 
 $c_0=c_0(\de,c_3)>0$.

Suppose now that $f_0$ decays instead exponentially as in part \eqref{ass:f-exp} of the assumption. We hence know that $-f_0'(v_\mathrm{max})= \al c_3 e^{-\alpha (v_\mathrm{max}-\La)}$. If $-\thalf\ell^{-1}f_0'(v_\mathrm{max})\geq \ell^{-\half}$, then \eqref{est:vmax-main} yields 
$$v_\mathrm{max}\geq 2\al c_3 \ell^{-1}e^{-\alpha (v_\mathrm{max}-\La)}=c e^{-\al v_\mathrm{max}}\ell^{-1}, \text{ for some } c>0.$$
This implies in particular that 
$e^{2\al v_\mathrm{max}}\geq \tilde c \ell^{-1}$ for some $\tilde c>0$, which, for sufficiently small $\ell_1$, yields the  claim. Otherwise $-f_0'(v_\mathrm{max})\leq 2\ell^{\half}$ which again yields that $v_\mathrm{max}\geq c\log(\ell^{-1})$ for some $c>0$ that depends only on $\alpha, \La,c_3$, provided $\ell_1$ is sufficiently small. 
\end{proof}

We can now prove the first part of Theorem \ref{thm:cyl} as follows.
Let $\de\in (0,1)$ be any fixed number and let $f = f_0(\cdot-\bar v)$ for $f_0$ as in  Assumption \ref{ass:f} with polynomial decay \eqref{ass:f-poly} (for some $c_3,\La>0$). 
We first observe that the sectional curvature of $(N,g_N)$ is bounded by a constant $\bar \ka$ which is independent of $\bar v$ and recall that also the numbers $\ell_1>0$ and $c_0>0$ obtained in Lemma \ref{lemma:leash_length} do not depend on $\bar v$.
We furthermore fix $E_0 = 10\pi$ and $\eps_0 = 1$ and let now  $\bar \ell>0$ be the number obtained in 
Theorem \ref{thm:1} 
corresponding to this choice of $\bar\kappa$, $E_0$, $\de$ and $\eps_0$, where we stress that $\bar \ell$ is independent of $\bar v$ which we have yet to choose.

Theorem \ref{thm:1} hence assures that any  solution of \eqref{eq:flow} that satisfies the assumptions \eqref{ass:ell_0} (for this $\bar \ell$) and \eqref{ass:tens_gives_length} of the theorem must degenerate in finite time.
We now finally choose the parameter $\bar v$ by which we shift the warping function so that $\ell_2(\bar v)\leq \bar \ell$ holds true for the number $\ell_2(\bar v)=\frac{5\pi^2}{\bar v^{2}}$ obtained in Lemma \ref{lemma:v-initial-bound}. Choosing our initial data $(u_0,g_0)$ as in Lemma \ref{lemma:initial}, we obtain from Lemma \ref{lemma:v-initial-bound} that the resulting solution $(u,g)$ satisfies the first assumption of Theorem \ref{thm:1} for every time in the maximal existence interval $[0,T)$. This lemma also assures that \eqref{est:v-centre} is satisfied for every $t\in[0,T)$ while Lemma \ref{lemma:energy-centre} implies that \eqref{est:energy-centre} holds true. We can thus apply the first estimate \eqref{eq:leash_delta_1} from Lemma \ref{lemma:leash_length} to conclude that also the second assumption of the theorem is satisfied. Hence Theorem \ref{thm:1} implies that the solution of Teichm\"uller harmonic map flow for any such initial data must degenerate in finite time as claimed in the first part of Theorem \ref{thm:cyl}.

We finally construct a compact target for which solutions of \eqref{eq:flow-resc} degenerate in finite time. For this we will modify the construction of Topping \cite{T-winding} 
so that the resulting warped product 
$$(N_c,g_{N_c})=(T^2,\gamma)\times_F (\widehat N, g_{\widehat N})$$
contains a 
totally geodesic submanifold $\Si$ which is isometric to $[0,\infty)\times_f (\widehat{N}, g_{\widehat N})$, for a coupling function $f=f_0(\cdot -\bar v)$, $\bar v\geq 7$, where we now ask that $f_0$ is as in Assumption \ref{ass:f} with exponential decay \eqref{ass:f-exp} for $\al=2\pi$, see below for the precise choice of $(N_c,g_{N_c})$.  Here $(\widehat N, g_{\widehat N})$ is the 3-torus obtained by identifying opposite faces of a suitably large cuboid in $(\Nt,g_\Nt)$ around the origin. We note that alternatively one could also construct such a warped product with a sphere as done in \cite{T-winding} to obtain a target for which Theorem \ref{thm:cyl} holds.

To obtain the existence of a suitable solution of the flow \eqref{eq:flow-resc} we first note: 

\begin{rmk} \label{rmk:flow-rescaled}
For any $\ell\in(0,\infty)$ there exists a harmonic map $u_\ell$ from $([-1,1]\times S^1,G_\ell)$ to $\Si$ which satisfies the assumptions of Lemma \ref{lemma:leash_length}. Indeed such a map can e.g.~be obtained as a limit as $t\to \infty$ of a solution of the 
 classical harmonic map flow for an initial map $u_0$ chosen as in Lemma \ref{lemma:initial}. 
\end{rmk} 
 This allows us to obtain a solution $(u_{\ell(t)},G_\ell(t))$ of the rescaled flow \eqref{eq:flow-resc} for which we can now argue  exactly as in the proof of the first part of Theorem \ref{thm:cyl}
to conclude that it must degenerate in finite time, except that we  
now  apply Theorem \ref{thm:2} instead of Theorem \ref{thm:1} (and of course choose $\bar \ell$ accordingly) and that we use the second estimate \eqref{eq:leash_delta_2} of Lemma \ref{lemma:leash_length} instead of the first estimate \eqref{eq:leash_delta_1}  of that lemma. As $\Si$ is a totally geodesic submanifold of our compact target this hence yields an example of a solution to \eqref{eq:flow-resc} into a compact target that degenerates in finite time as claimed in the second part of Theorem \ref{thm:cyl}.

We finally explain how such a target can be chosen: 
 The torus component $(T^2,\gamma)$ of $(N_c,g_{N_c})$ is exactly as in  \cite{T-winding}, 
and, as in \cite{T-winding}, we consider the submanifold $\Si=\{ z=\frac{1}{w}, w\in(0,1]\}\times_F (\Nth,g_\Nth)$, which we recall is totally geodesic if the coupling function $F$ satisfies $\partial_x F \vert_{\{x=0\}}=0$ with respect to the coordinates $x=\frac{1}{w}-z$, $y=z$ used in \cite{T-winding}. 
We choose $F$ to be of the form 
\beqs 
F(x,y) =  h(x+y-\bar v)e^{-2\pi(x+y-\bar v-\La)} \big(\sin 2\pi \big(x - \tfrac18 \big) + \sqrt{2} \, \big) + k(x+y-\bar v)  + 1, \eeqs
for smooth functions $h,k\geq 0$, 
where we note that the case $h\equiv 1$ and $k\equiv 0$ would correspond to the coupling function used in \cite{T-winding}.
For our construction, we need
$\partial_x F\vert_{\{x=0\}} =0$, i.e. 
\beq
\label{eq:tot-geod}
k'(v)+e^{-2\pi(v-\La)}\tfrac{\sqrt{2}}{2}h'(v)=0,
\eeq
as well as that $f_0(v):=F(0,v+\bar v)=h(v)e^{-2\pi(v-\La)}\frac{\sqrt{2}}{2}+k(v)+1$ satisfies Assumption \ref{ass:f} with decay \eqref{ass:f-exp}. 
Suitable functions $h,k$ can e.g.~be obtained as follows:
We first select $k$ so that for $c_4$ chosen below and some $\La=\La(c_4)>0$ we have $k\equiv 7$ on $(-\infty,0]$, $k\equiv 0$ on $[\La,\infty)$ and so that the derivative of $k$ satisfies 
$0\leq -k'\leq e^{\pi}c_4$ on all of $\R$ with $-k'\equiv e^{\frac\pi2} c_4$ on $[\frac12,\frac34]$, $-k'\equiv \frac14 c_4$ on $[1,\La-1]$ as well as $-k''\leq 0$ on $[1,\La]$.

We then set $h(v):=\sqrt{2}\int_0^ve^{2\pi(t-\La)}\abs{k'(t)} dt$, which ensures not only that \eqref{eq:tot-geod} is satisfied, but also that $f_0$ has the behaviour on $(-\infty,0]$ and $[\La,\infty)$ required in Assumption \ref{ass:f}.
As \eqref{eq:tot-geod} furthermore implies that $-f_0'(v)=\sqrt{2}\pi e^{-2\pi(v-\La)}h(v)\geq 0 $ and as a short calculation shows that for $c_4$ sufficiently small also $-f_0'(v)\leq \frac18$ for all $v\in\R$, it remains to check that $-f_0''\leq 0$ on $[1,\infty)$ which is equivalent to $h'\leq 2\pi h$. This can be easily seen as the choice of $k$ ensures that $2\pi h(1)\geq h'(1)$ and that $2\pi h'-h''\geq 0$ on $[1,\infty)$.

\appendix

\section{Appendix}
\subsection{Teichm\"uller harmonic map flow from cylinders}
\label{appendix:cyl}
Here we recall the key features of the definition of the flow for cylindrical domains, see \cite{R-cyl} for more details. 
We recall that horizontal curves of metrics on $\Cyl=[-1,1]\times S^1$ are described in terms of hyperbolic metrics $G_\ell$ which are given by 
\beq \label{def:Gell}
G_\ell=f_\ell^{*}\big(\rho_\ell^2(ds^2+d\th^2)\big), \ f_\ell \colon \Cyl\to [-X(\ell),X(\ell)]\times S^1 \text{ for } X(\ell)=\tfrac{2\pi}{\ell}(\tfrac\pi2-\arctan(\tfrac\ell d)) 
\eeq 
where $f_\ell$ are the smooth diffeomorphisms defined in Lemma 2.4 of \cite{R-cyl}.
Here $d > 0$ is a fixed number which corresponds to the choice of uniformisation one uses to represent conformal classes by hyperbolic metrics with boundary curves of constant geodesic curvature, c.f.~\cite{R-boundary}, and can be characterised by 
\beq d = k_{G_\ell} |_{\partial M} \cdot L_{G_\ell}(\{\pm1\}\times S^1) \text{ or equivalently by }
d^2 = L_{G_\ell}(\{\pm1\}\times S^1)^2 - \ell^2 . \eeq
Horizontal curves of metrics are then given by one-parameter families of metrics of the form $G_{\ell(t)}$, in general pulled back by a fixed diffeomorphism.

While for Teichm\"uller harmonic map flow from closed surfaces one evolves the metric orthogonally to all diffeomorphisms and hence by a horizontal curve, for the flow from cylinders one needs to impose a three-point condition for the map component on both boundary curves. This forces us to additionally evolve the metric component of the flow in the direction of a six-parameter family $h_{b,\phi}$ of diffeomorphisms, whose restrictions to the boundary $h_{b,\phi}|_{\{\pm1\}\times S^1}$ are given by M\"obius transforms.
As we can assume without loss of generality that $g(0)$ is  of the form $h_{b_0,\phi_0}^*G_{\ell_0}$, 
the metric component of the flow is hence given by a curve $g(t) = h_{b,\phi}^*G_{\ell(t)}$ that evolves by \eqref{eq:flow} respectively \eqref{eq:flow-resc}, now with $P_g$ denoting the $L^2$-orthogonal projection onto the tangent space of the set 
 of admissible metrics $\{h_{b,\phi}^*G_\ell\}$. This tangent space is given by the $L^2$-orthogonal sum
$\Rea(\Hol(g)) \oplus \{ L_{(h_{b,\phi}^*\!X)} g : X \in \mathcal X \}, $
 where $\mathcal X$ is the six-dimensional space of vector fields generating the diffeomorphisms $h_{b,\phi}$, and  $\Hol(g)$ is the space of  holomorphic quadratic differentials on the cylinder which are real on the boundary. As a result, for the flows \eqref{eq:flow} and \eqref{eq:flow-resc} from cylinders, the evolution on $\ell$ is determined by the projection of the real part of the Hopf differential $\Phi(u,g)$ onto $\Rea(\Hol(g))$. 
In the special case that  $\Phi$ itself is holomorphic in the interior of the cylinder
as encountered in the proof of Theorem \ref{thm:2}, we hence simply have $\frac{d\ell}{dt}=-\frac{2\pi^2}{\ell}\Rea(b_0(\Phi))$, $b_0$ the principal part of the Fourier expansion of $\Phi$.

\subsection{Properties of hyperbolic collars}
\label{appendix:collar}
In this appendix we collect well-known results on hyperbolic collars that are used throughout the paper, including the classical collar lemma
\begin{lemma}[Keen--Randol \cite{randol}] \label{lemma:collar}
Let $(M,g)$ be a closed oriented hyperbolic surface and let $\si$ be a simple closed geodesic of length $\ell$. Then there is a neighbourhood $\Col(\si)$ around $\si$, a so-called collar, which is isometric to the 
cylinder 
$[-X(\ell),X(\ell)]\times S^1$
equipped with the metric $\rho^2(s)(ds^2+d\theta^2)$ where 
$\rho(s)=\rho_\ell(s)=\frac{\ell}{2\pi} \cos(\frac{\ell s}{2\pi})^{-1} $
 and
$X(\ell)=\frac{2\pi}{\ell}\big(\frac\pi2-\arctan\big(\sinh\big(\frac{\ell}{2}\big)\big) \big).$
\end{lemma}

Throughout the paper we use the following estimates on hyperbolic cylinders with $\ell\in (0,\arsinh(1))$ as considered in \eqref{eq:collar-general} and \eqref{def:constants-collar} that hold true for constants that depend at most
on the numbers $\bar c_{1,2}$ from \eqref{def:constants-collar}.
We have a uniform upper bound on $\rho(s)$, say $\rho(s)\leq C_3$, and thus also on $\abs{\partial_s\log(\rho))}\leq \rho$ allowing us in particular to bound 
\beq \label{est:rho-different-points-1}
\rho(s)\leq e^{C_3 \abs{s-s'}} \rho(s') \text{ for any } s,s'\in [-X(\ell),X(\ell)].
\eeq
For points for which $\rho$ is small we use instead that  
$\abs{\partial_s\log(\rho))}\leq \rho$ also implies that  
\beq \label{est:rho-different-points-2}
e^{-1} \cdot \rho(s_0)\leq \rho(s)\leq e \cdot \rho(s_0) \text{ for all } s,s_0\in [-X(\ell),X(\ell)] \text{ with } \abs{s-s_0}\leq e^{-1}\rho^{-1}(s_0).
\eeq
Indeed, if for any $s_0$ 
the maximal number $h>0$ for which the above estimate holds true on the interval $I=[s_0,s_0+h]$ were less than $\min(e^{-1}\rho^{-1}(s_0), X-s_0)$, then 
$1=\abs{\log(\rho(s_0+h))-\log(\rho(s_0))}\leq h \sup_I\rho\leq h e \rho(s_0)<1$ would lead to a contradiction.

We shall also use that for any $0\leq 2\ell\leq \eps < \inj_g(X(\ell))$ the $\eps\thin$ part of such a hyperbolic collar is described in collar coordinates by $(-X_\eps(\ell),X_\eps(\ell))\times S^1$, where 
\beq \label{eq:X-delta}
X_\eps:= \tfrac{2\pi}{\ell}\big[\tfrac{\pi}{2}-\arcsin \big(\sinh(\tfrac{\ell}{2})\cdot (\sinh \eps)^{-1}\big)] \text{ and hence so that } \frac{C}{\eps}\geq X(\ell)-X_\eps(\ell)\geq \frac{c}{\eps} 
\eeq
for some $C,c>0$. 
A short calculation furthermore shows that 
\beq \label{est:int_rho_by_ell} \int_0^{X(\ell)} \rho^\frac{1}{2} \leq 2\sqrt{2}\pi \, \ell^{-\frac{1}{2}} \text{ and }
\Area([0,s]\times S^1)=2\pi \int_0^s\rho^2 \leq 2\pi \rho(s).
\eeq

{\sc CR \& MR: 
Mathematical Institute, University of Oxford, Oxford, OX2 6GG, UK

\end{document}